\documentclass[11pt,letterpaper]{article}

\usepackage[margin=1in]{geometry} 
\usepackage{amsmath,amsthm,amssymb,enumitem, bbm, amsthm}
\usepackage{graphics}
\usepackage{tikz}
\usepackage[dvipsnames]{xcolor}
\usepackage{hyperref}
\hypersetup{
	colorlinks=true,
	linkcolor=blue,
	citecolor=blue,
	filecolor=magenta,      
	urlcolor=cyan,
	linktoc=page
}

﻿
﻿
﻿
﻿
\newcommand{\veeSymbol}{
	\begin{tikzpicture}[baseline]
		\draw[scale=0.2, thick] (0,0) -- (1,1);   
		\draw[scale=0.2, thick] (0,0) -- (-1,1);  
	\end{tikzpicture}
}
﻿
﻿
\newcommand{\wedgeSymbol}{
	\begin{tikzpicture}[baseline]
		\draw[scale=0.2, thick] (-1,0) -- (0,1);   
		\draw[scale=0.2, thick] (0,1) -- (1,0);    
	\end{tikzpicture}
}
﻿
﻿
﻿
﻿

\newcommand{\R}{\mathbb{R}}  
\newcommand{\Z}{\mathbb{Z}}
\newcommand{\N}{\mathbb{N}}

\newcommand{\prob}{\mathbb{P}}
\newcommand{\Lag}{\mathcal{L}}

\newcommand{\E}{\mathbb{E}}

\newcommand{\Ham}{\mathbf{H}}
\newcommand{\ext}{\textrm{ext}}
﻿
﻿
﻿
﻿
\theoremstyle{plain}
\newtheorem{thm}{Theorem}[section]
﻿
﻿
\theoremstyle{definition}
\newtheorem{defn}[thm]{Definition} 
\newtheorem{lemma}[thm]{Lemma} 
\newtheorem{cor}[thm]{Corollary} 
\newtheorem{con}[thm]{Conjecture} 
\newtheorem{remark}[thm]{Remark} 
﻿
﻿
\numberwithin{equation}{section}
﻿
﻿
﻿
﻿
\newcommand*{\defeq}{\mathrel{\vcenter{\baselineskip0.5ex \lineskiplimit0pt
			\hbox{\scriptsize.}\hbox{\scriptsize.}}}%
	=}

\begin{document}
	
	\title{Bulk heights of the KPZ line ensemble}
	\author{Duncan Dauvergne \and Fardin Syed}
	
	\maketitle
	
	\begin{abstract}
		For $t > 0$, let $\{ \mathcal{H}^{(t)}_n, n \in \N\}$ be the KPZ line ensemble with parameter $t$, satisfying the homogeneous $\Ham$-Brownian Gibbs property with $\Ham(x) =e^x$. We prove quantitative concentration estimates for the $n$th line $\mathcal{H}^{(t)}_n$ which yield the asymptotics $\mathcal{H}^{(t)}_n = n \log n + o(n^{3/4 + \epsilon})$ as $n \to \infty$. A key step in the proof is a general integration by parts formula for $\Ham$-Brownian Gibbs line ensembles which yields the identity $\E \exp(\mathcal{H}^{(t)}_{n + 1}(x) - \mathcal{H}^{(t)}_n (x)) = n t^{-1}$ for any $n, t, x$. 
	\end{abstract}

	\section{Introduction}
	
	Introduced by Kardar, Parisi and Zhang \cite{PhysRevLett.56.889}, the KPZ equation 
	\begin{align} \label{eqn:kpz_eqn}
		\partial_t \mathcal{H} = \frac{1}{2} \partial_x^2 \mathcal{H} + \frac{1}{2} (\partial_x \mathcal{H})^2 + \xi
	\end{align}	
	is a fundamental model for random one-dimensional interface growth. Here $\mathcal H = \mathcal{H}(t,y)$ denotes the height of the interface at time $t$ and spatial position $y$, and $\xi$ is a space-time white noise. The KPZ equation is often understood through a Cole-Hopf transformation of the multiplicative stochastic heat equation (SHE). That is, the solution to the KPZ equation is given by $\mathcal{H}(t,x) = \log \mathcal{Z}(t,x)$ where $\mathcal{Z}$ satisfies the following stochastic PDE:
	\begin{align} \label{eqn:SHE}
		\partial_t \mathcal{Z} &= \frac{1}{2} \partial_y^2 \mathcal{Z} + \mathcal{Z} \xi.
	\end{align}
	While it is easy to check the formal relationship between \eqref{eqn:kpz_eqn} and \eqref{eqn:SHE}, making this rigorous is extremely difficult and was finally understood in \cite{hairer2012solvingkpzequation}.
	
	The KPZ equation with \textit{narrow-wedge initial data} is given by $\mathcal{H}^{\mathrm{nw}} = \log \mathcal{Z}^{\mathrm{nw}}$, where $\mathcal{Z}^{\mathrm{nw}}$ corresponds to the solution of equation \eqref{eqn:SHE} with Dirac-delta initial condition $\mathcal{Z}(0,y) = \delta_0(y)$. At a fixed time $t > 0$, the function $x \mapsto \mathcal{H}^{\mathrm{nw}}(x, t)$ can be embedded as the top curve $\mathcal H^{(t)}_1$ in a random family of functions $\mathcal H^{(t)} :=\{\mathcal H^{(t)}_n:\R\to \R, n \in \N\}$ known as the \textbf{KPZ$_t$ line ensemble}. The KPZ$_t$ line ensembles were constructed in \cite{corwin2016kpz} as edge-scaling limits of the O’Connell–Yor polymer. An alternate construction can be given through a multi-layer extension of the multiplicative SHE \cite{O_Connell_2015}. The two definitions were connected in \cite{Nica_2021}. A third definition can be derived from \cite{dimitrov2022characterization}: the KPZ$_t$ line ensemble is the unique (in law) family of curves whose lowest indexed line $\mathcal H^{(t)}_1$ is equal in law to $\mathcal{H}^{\mathrm{nw}}(\cdot, t)$ and which satisfies a certain $\Ham$-Brownian Gibbs property, see \eqref{E:RN-intro} and surrounding discussion.
	
	 Because slightly different scalings of the KPZ line ensembles appear in the literature, we adopt in this paper the precise definition given in Section~\ref{subsec:KPZ_Line_Ensemble_defn}, which coincides with the formulation in \cite{Nica_2021} and with the originally published version of \cite{corwin2016kpz}. This choice of scaling ensures that the KPZ$_t$ line ensemble possesses the \textit{homogeneous} $\Ham$-Brownian Gibbs property and is stationary with respect to the parabola $-x^{2}/(2t)$, see Section \ref{S:comparison} for discussion.
	
	Since its introduction in \cite{corwin2016kpz}, the KPZ$_t$ line ensembles have been used to derive a rich collection of results about the KPZ equation, e.g. see \cite{das2023law, corwin2021kpz, ganguly2023brownian, ganguly2022sharp, corwin2020kpzequationtailsgeneral}. Most prominently, as $t \to \infty$, a combination of \cite{wu2023convergence, virag2020heat, dimitrov2021characterization, aggarwal2023} implies that the KPZ$_t$ line ensemble converges after rescaling to the Airy line ensemble \cite{prahofer2002scale, corwin2013browniangibbspropertyairy}, a non-intersecting line ensemble which arises as a universal scaling limit in random matrices, random tilings, and random interface growth, e.g., see \cite{johansson2003discrete, baik2007discrete, petrov2014asymptotics, duse2018universal, dauvergne2023uniform, aggarwal2025edge} for a sample of results and the survey \cite{johansson2018edge} for more references. Wu \cite{wu2025kpzequationdirectedlandscape} used this to prove convergence of the full set of solutions to the KPZ equation driven with the same noise to the directed landscape \cite{Dauvergne_2022}, the richest limit object in the KPZ universality class. 
	More recently, \cite{wu2025optimaltransport} showed that any solution to the KPZ equation at time $t$ is an (explicit) measurable function of the KPZ$_t$ line ensemble. 
		
		However, despite the central role played by the object, the behaviour in the bulk of the KPZ line ensembles (i.e., when the line index $n$ is large) is still poorly understood. Our main result, Theorem~\ref{thm:main_thm}, is a quantitative bound on the location of the KPZ$_t$ line ensemble as the index $n$ tends to infinity. For this theorem and throughout the paper, for $t > 0$, $n \in \N$, we define 
	\begin{align*}
		e_n^{(t)} &= \log (t^{1 - n} (n-1)!) + \frac{t}{24}.
	\end{align*}
	\begin{thm} \label{thm:main_thm} 
		Fix $t_0 > 0, \epsilon > 0$.  There exist constants $d, C > 0$ so that for all $n \in \N, t > t_0$ and $a \ge C (t n^{2\epsilon} + t^{1/4} n^{3/4 + 2\epsilon})$ we have
	\begin{align*}
		\prob \left(  | \mathcal{H}^{(t)}_n(0) - e_n^{(t)} | > a \right) & \leq 2 \exp(- d a n^{\epsilon}  \max(1, t^{-1/4} n^{1/4})).
	\end{align*}
	\end{thm}
	
	In particular, Theorem \ref{thm:main_thm} implies that the KPZ$_t$ line ensemble satisfies the following strong law of large numbers. Almost surely,
	\begin{equation}
		\lim_{n \to \infty} \frac{\mathcal{H}^{(t)}_{n}(0)}{n \log n} = 1.
	\end{equation}
	In particular, the KPZ line ensembles tend to \textit{positive} $\infty$ as we increase the line index! This behavior is in stark contrast to the Airy line ensemble $\{\mathcal{A}_n\}_{n=1}^{\infty}$, which diverges to \textit{negative} $\infty$: $\mathcal{A}_n(0) = (1 + o(1))(-3 \pi n/2)^{2/3}$, see Section \ref{S:comparison} for further comparison.
	
	Theorem \ref{thm:main_thm} only gives asymptotics for one-point statistics of KPZ lines. Given the exponential control in Theorem \ref{thm:main_thm}, we can easily pass from one-point statistics to supremum/infimum bounds on the KPZ line ensemble using a strong modulus of continuity for the KPZ line ensemble \cite{wu2025optimaltransport} and parabolic stationarity. Below is a sample statement.
	
	\begin{cor}
\label{C:mod-plus-moment}
	Fix $t_0 > 0, \epsilon > 0$. Then there exist constants $d, C > 0$, such that for all $n \in \N, t > t_0$ and $a > C (t n^{2\epsilon} + t^{1/4} n^{3/4 + 2\epsilon})$ we have the estimate
\begin{align*}
	\prob \left(  \sup_{x \in [-e^a, e^{a}]} | \mathcal{H}^{(t)}_n(x) - \frac{x^2}{2t} - e_n^{(t)} | > a \right) & \leq 2 \exp(- d a n^{\epsilon}  \max(1, t^{-1/4} n^{1/4})).
\end{align*}
	\end{cor}
	
	\begin{remark}
		\label{R:small-n}
		Our focus in Theorem \ref{thm:main_thm} is in obtaining interesting bounds in the large $n$ regime. Indeed, while we do not expect that any of our estimates are optimal, the bounds for small $n$ and large $t$ are particularly weak. We note in passing that as part of the proof of Theorem \ref{thm:main_thm} we obtain the following one-sided estimate, which gives better control for certain values of $a$ when $t$ is small. For all $t_0 > 0$, there exists a constant $d > 0$ such that for any $a > 0, t > t_0$ and $n \in \N$, we have 
		$$
		\prob\left(  \mathcal{H}^{(t)}_n(0) - e_n^{(t)}  > a \right) \le 2\exp(-d a^{3/2} t^{-1/2}) + n \exp(-a/(2n)).
		$$
		This estimate is contained in Lemma \ref{lemma:endpts_expmm}. By combining the above bound with Theorem \ref{thm:main_thm}, we can see that the KPZ$_t$ line ensemble \textit{turns around}. For example, for an index $i \le t^{1 - \epsilon}$ the above bound and the known asymptotics for $\mathcal H_1^{(t)}$ imply that 
		$$
		\mathcal H_1^{(t)}(0) - \mathcal H_i^{(t)}(0) \ge (i/2) \log t
		$$
		with probability tending to $1$
		as $t \to \infty$. On the other hand, Theorem \ref{thm:main_thm} implies that for $i \ge t^{1 + \epsilon}$ we have that
		$$
		\mathcal H_1^{(t)}(0) - \mathcal H_i^{(t)}(0) \le - (i/2) \log i
		$$
		with probability tending to $1$ as $t \to \infty$.	
	\end{remark}
	
	\begin{remark}[Future work] One of the main motivations behind our bounds in Theorem \ref{thm:main_thm} is to prove a new characterization of the KPZ sheet from time $0$ to time $t$ in terms of the KPZ line ensemble \cite{dauvergne2026+}, based on a version of the geometric Robinson-Schensted-Knuth isometry (see \cite{corwin2021invariance} and \cite[Theorem 3.3]{dauvergne2022hidden}) which comes from extension of a discrete result in \cite{noumi2004tropical}. It turns out that the characterization requires an asymptotic bound on the KPZ line locations as the crucial technical estimate.
	\end{remark}
	\subsection{Should we be surprised? A discussion of Theorem \ref{thm:main_thm}}
	\label{S:comparison}
	
	Before moving into the proof of Theorem \ref{thm:main_thm}, we give two heuristics for understanding the bulk behaviour of the KPZ line ensemble.
	
	A good starting point for understanding the KPZ line ensemble is to first consider its more accessible cousin, the (parabolic) Airy line ensemble. The parabolic Airy line ensemble is a collection of non-intersecting paths $\mathcal A = \{ \mathcal A_{n}:\R \to \R, n \in \N \}$ satisfying the ordering constraint $\mathcal A_{i} > \mathcal A_{i + 1}$ for all $i \ge 1$. The process $x \mapsto \mathcal A(x) + x^2$ is stationary. The Airy line ensemble was first constructed in \cite{prahofer2002scale} as a scaling limit of the polynuclear growth model, and realized as a system of non-intersecting locally Brownian paths in \cite{corwin2013browniangibbspropertyairy}. 
	
	The paper \cite{corwin2013browniangibbspropertyairy} shows that the Airy line ensemble satisfies a remarkably useful resampling property, known as the \textbf{Brownian Gibbs property}. For any finite rectangle $R := [a, b] \times \{1, \dots, k\} \subset \mathbb{N} \times \mathbb{R}$, the conditional law of $\{\mathcal{A}_i(x) : (x, i) \in R\}$ given the values of $\mathcal{A}$ on $R^c$, is that of a collection of $k$ Brownian bridges connecting up the endpoints $\mathcal{A}_i(a), \mathcal{A}_i(b), i = 1, \dots, k$, and constrained so that all paths are non-intersecting. In fact, this property and the parabolic shape of the top line together characterize the object, see \cite{aggarwal2023}. 
	
Like the Airy line ensemble, the KPZ$_t$ line ensemble $\mathcal H^{(t)}$ is also stationary under a parabolic shift: $x \mapsto \mathcal H^{(t)}(x) + x^2/(2t)$. Moreover, it satisfies a soft version of the Brownian Gibbs property, known as an $\Ham$-Brownian Gibbs property. For a rectangle $R$ as above, given the values of $\mathcal H^{(t)}$ on $R^c$, the conditional law of $\mathcal H^{(t)}$ on $R$ is that of $k$ Brownian bridges connecting up the endpoints $\mathcal{H}_i^{(t)}(a), \mathcal{H}_i^{(t)}(b), i = 1, \dots, k$, reweighted by a Radon-Nikodym derivative proportional to
	\begin{align}
	\label{E:RN-intro}
\exp \big( - \mathcal{E}_{\Ham} (B) \big) := \exp\left( - \sum_{i = 1}^{n} \int_a^b \Ham( B_{i+1}(x) - B_i(x))dx \right), \qquad \Ham(x) = e^x.
\end{align}
Notice that this Radon-Nikodym derivative provides a double exponential penalty whenever two lines in the ensemble violate the ordering $\mathcal H_i >  \mathcal H_{i+1}$, and so while lines are not strictly ordered in the KPZ$_t$ line ensembles, the formula \eqref{E:RN-intro} suggests that they strongly prefer to be that way. Further strengthening the connection between the KPZ$_t$ and Airy line ensembles is the scaling limit result of \cite{wu2023convergence, virag2020heat, dimitrov2021characterization, aggarwal2023} discussed above. If $ \{ \mathcal{H}_n^{(t)} \}_{n = 1 }^{\infty} $ denotes the KPZ$_t$ line ensemble, then
\begin{align*}
	\Big\{ t^{-1/3} \mathcal{H}_n^{(t)}(t^{2/3} x) + \frac{t^{4/3}}{24} \Big\}_{n \in \mathbb{N}} \xrightarrow[t \to \infty]{d} \mathcal A
\end{align*}
almost surely on compact subsets of $(x, n) \in \mathbb{R} \times \mathbb{N}$. To see why we might expect the above convergence, first note that the rescaling of $\mathcal H^{(t)}$ on the left-hand side above is stationary after addition of the parabola $x^2$, and satisfies an $\Ham$-Brownian Gibbs property with $\Ham(x) = e^{t^{1/3} x}$. In the $t \to \infty$ limit, this Gibbs property becomes the usual Brownian Gibbs property, which can equivalently be defined using \eqref{E:RN-intro} with $\Ham(x) = \infty \mathbf{1}(x \ge 0)$.

Given the many parallels between the KPZ and Airy line ensembles, it may be quite surprising that their bulk behaviour differs so dramatically. Indeed, it is not immediately clear why the Gibbs property \eqref{E:RN-intro}, which seems to encourage the line ordering $\mathcal H_1 > \mathcal H_2 > \cdots$ should result in the opposite ordering asymptotically in $n$. At the highest level, we can justify this by imagining that unlike the Brownian Gibbs property, the soft $\Ham$-Brownian Gibbs property \eqref{E:RN-intro} is too weak to compensate for the downward parabolic pull on the lines without the smallest indexed lines sagging.

An alternate point of view on the KPZ line ensembles comes through directed polymers. At the level of the KPZ equation, the connection to a directed polymer is given in \cite{alberts2014continuum}. The extension to the KPZ line ensemble comes from the  
the multi-layer stochastic heat equation \cite{O_Connell_2015}, which we describe formally. 
Let $\xi = \xi(t, x)$ be a space-time white noise on $\R^2$. 
For $t > 0$, define
\begin{equation}
	\label{E:wick}
	\mathcal{Z}^t_n(x) = p(t, x)^{n} \, \mathbb{E}_{0^n}^{x^n} \Big[ : \exp : \Big( \sum_{i=1}^n \int_0^t \xi(s, B_i(s)) \, ds \Big) \Big],
\end{equation}
where $p(t, x) = (2 \pi t)^{-1/2} e^{-x^2/2t}$ is the Gaussian transition kernel and $\mathbb{E}_{0^n}^{x^n}$ is the distribution of $n$ non-intersecting Brownian bridges $B_1, \ldots, B_n$ each starting at location $0$ at time $0$ and ending at location $x$ at time $t$. Since $\xi$ is defined only as a distribution, the integrals $\int \xi(s, B_i(s)) \, ds$ are ill-defined, and thus need to be renormalized correctly. Here $: \exp :$ is the Wick exponential, which effectively is a shorthand for rigorously interpreting \eqref{E:wick} through the following Wiener chaos expansion:
\begin{align*}
	Z^{t}_n(x) \defeq p(t, x)^n  \Big(1 +  \sum_{k = 0}^{\infty} \int_{0 < t_1 < \cdots t_k < t} \int_{\R^k} R^{(n)}_k\big( (t_1, x_1) , \ldots , (t_k, x_k)  \big) \xi(dt_1, dx_1) \ldots \xi(dt_k , dx_k) \Big).
\end{align*}
Here $R^{(n)}_k$ is the $k$-point correlation function for $n$ non-intersecting Brownian bridges started at location $0$, time $0$ and ending at location $x$, time $t$. 
The intuition behind \eqref{E:wick} is that $Z^t_n$ is the partition function for a multi-path Brownian polymer in a white noise background. This point of view is further strengthened by the fact that we can alternate define $Z^t_n$ by taking a limit of the partition function for non-intersecting semi-discrete or lattice polymers, see \cite{Nica_2021, corwin2017intermediate} and Section \ref{subsec:KPZ_Line_Ensemble_defn}.

Just as the KPZ equation is connected to the stochastic heat equation, the KPZ line ensembles are connected to the multi-layer stochastic heat equation. The precise relationship was established by \cite{Nica_2021}. With our choices of normalization, we have that
\begin{align}
	\label{E:Htn}
	\mathcal{H}^{(t)}_n = \log \big( \frac{\mathcal{Z}^t_{n} }{\mathcal{Z}^t_{n - 1}}  \big) + \log (t^{1 - n} (n - 1)!).
\end{align}
The constant factor in \eqref{E:Htn} is, of course, the diverging piece of $e_n^{(t)}$. The remaining piece of $\mathcal{H}^{(t)}_n$ can be viewed as the (multiplicative) gain in the partition function $\mathcal{Z}^t_{n}$ as we add an extra path to our non-intersecting continuum polymer. If we have executed the normalization of $\mathcal{Z}^t_{n}$ in a nice way, it is reasonable to expect that this gain is of lower order, which greatly removes the mystery behind Theorem \ref{thm:main_thm}. However, while \eqref{E:Htn} gives a clear heuristic, we do not know how to prove Theorem \ref{thm:e_moments} from the polymer description and in the proof we will only ever work directly with $\mathcal{H}^{(t)}$. In particular, the factor $e_n^{(t)}$ in our theorem is derived from a connection to the one-dimensional Toda lattice, whereas in the derivation of \eqref{E:Htn} in \cite{Nica_2021} this constant comes from the volume of the $n$-dimensional simplex.

	\subsection{A more promising heuristic, the 1d Toda Lattice, and a proof sketch}

	To begin to understand the bulk behaviour of the KPZ$_t$ line ensembles, first consider the following resampling problem in either the Airy or KPZ$_t$ line ensembles. 
	\begin{description}
		\item[Parabolic boundary problem.] Fix the $(n+1)$st curve of a line ensemble to lie along the deterministic parabola $-x^2/(2t)$ on a large interval $[-T, T]$. Sample $n$ curves above this lower boundary, using either the Brownian Gibbs property or the $\mathbf{H}$-Brownian Gibbs property with $\Ham(x) = e^x$, with appropriate boundary conditions at $\pm T$.
	\end{description}
	In the Brownian Gibbs setting, the strict ordering constraint forces each curve to lie above the one below it, effectively requiring all $n$ curves to stay above the parabola. Lifting $n$ Brownian paths above a parabolic barrier over a long interval carries a substantial entropic (or Wiener) cost, but the ordering of the curves is preserved. Nonetheless, curves of large index are compressed more because of this entropic cost, which forces the gaps between Airy lines to decrease with increasing index.
	
	In contrast, under the $\mathbf{H}$-Brownian Gibbs property, disorder among the curves is not strictly forbidden but instead penalized via a double-exponential energy. As a result, the top few lines may still remain ordered to avoid this penalty. However, for sufficiently large $n$, the cost of maintaining all curves above the parabola outweighs the penalty associated with a few disorder violations. Thus, beyond a certain index, it becomes energetically favourable for some curves to violate the ordering, leading to a breakdown of the monotonicity in the lower part of the ensemble.
	
	Since the KPZ$_t$ line ensemble is stationary with respect to the parabolic shift $-x^2/(2t)$, one can analyze the balance between the Wiener cost and the interaction energy imposed by the $\Ham$-Brownian Gibbs property. This balance yields constraints on the distribution of gaps between successive curves. In particular, for a line ensemble governed by the $\mathbf{H}$-Brownian Gibbs property to be stationary under this parabolic shift, it turns out that the exponential moments of the inter-line gaps must take specific values.
	\begin{thm} \label{thm:e_moments}Let $\{ \mathcal{H}_{i}\}_{i = 1}^{\infty}$ be the KPZ$_t$ line ensemble. Then for all $i \ge 1$:
		\begin{align*}
			\E \big[ \exp(\mathcal{H}_{i +1}(0) - \mathcal{H}_i(0))\big] =  \frac{i}{t}.
		\end{align*} 
	\end{thm}
	Theorem \ref{thm:e_moments} is a special case of a more general balance-of-costs theorem for $\Ham$-Brownian line ensembles, recorded in Theorem \ref{thm:int_parts}. Theorem \ref{thm:int_parts} is proven by applying a Gaussian integration by parts argument to discretized ensembles, and then taking a limit.
	
	An immediate consequence of Theorem \ref{thm:e_moments} is that for $i \ll t$, we have $\mathcal{H}_i > \mathcal{H}_{i+1}$ with high probability, which is consistent with the strict ordering observed in the Airy line ensemble. However, when $i \gg t$, the exponential moment of the gap diverges, suggesting a breakdown of the ordering and the emergence of reverse monotonicity in the bulk of the ensemble.

	Given Theorem \ref{thm:e_moments}, it is easy to guess Theorem \ref{thm:main_thm}. Indeed, if we assume that all the lines $\mathcal H_i$ are deterministic, then Theorem \ref{thm:e_moments} forces $\mathcal H_i(0) = \log(t^{1-n}(n-1)!) + c$ for some constant $c$. We can establish that $c = t/24$ by the known asymptotics of $\mathcal H_1$. The difficulty in moving from Theorem \ref{thm:e_moments} to the more difficult Theorem \ref{thm:main_thm} is that an exponential moment can easily lie about the typical behaviour of a random variable.
	
	Let us now push further the problem of balancing the Wiener cost against the interaction energy in the parabolic boundary problem above. Recall that we can view the density of the Wiener measure of at a path $f$ as proportional to an exponentiated Dirichlet energy:
	\begin{align*}
		\mathcal{W}(f) \propto \exp\big( - \mathcal{E}_{\mathrm{dir}}(f) \big), \qquad \mathcal{E}_{\textrm{dir}} (f) = \sum_{i = 1}^n \int_{-T}^T \frac{1}{2} |\dot{f}|^2. 
	\end{align*}
	Consequently, the density of a $\mathbf{H}$-Brownian motion can be formally expressed as 
	$$
	\mathcal{P}_{\Ham}(f) \propto \exp\big(- \mathcal{E}_{\Ham}(f) - \mathcal{E}_{\mathrm{dir}}(f)\big).
	$$ 
	The most probable path under this measure corresponds to the minimizer of the combined energy functional 
	$$
	\mathcal{E}(f) = \mathcal{E}_{\Ham}(f) + \mathcal{E}_{\mathrm{dir}}(f).
	$$ 
	Applying a standard Euler–Lagrange variational principle to $\mathcal{E}$ reveals that this optimal path satisfies the equations of the one-dimensional Toda lattice
	\begin{align*}
		\ddot{f_i} &= e^{f_i - f_{i - 1}} - e^{f_{i + 1} - f_i }, \qquad f_0 = + \infty, \quad f_{n+1}(x) = - x^2/2t. 
	\end{align*}
	It is reasonable to expect that at large scales, paths in the parabolic boundary problem concentrate around an energy minimizer. Moreover, the parabolic stationarity of the KPZ$_t$ line ensemble suggests that we should have an energy minimizer whose paths are all shifted parabolas
		\begin{align*}
		f_i(x) = -\frac{x^2}{2t} + c_i, \qquad i = 1, \dots, n.
	\end{align*}
	Indeed, it is not hard to check that there is a unique such solution with the given values of $f_{n+1}, f_0$. The condition $\ddot{f}_i = -t^{-1}$ together with the Toda equations implies that
	\begin{align*}
		e^{f_{i+1} - f_i} = \frac{i}{t}, \qquad \text{or equivalently} \qquad c_i = -\sum_{j=i}^n \log(\frac{i}{t}).
	\end{align*}
	This matches the gap spacing in our main theorems. 
	
The proof of Theorem \ref{thm:main_thm} uses this Toda heuristic to bound the KPZ$_t$ line ensembles. This consists of two steps. In Section \ref{sec:toda_bdary}, we study the parabolic boundary problem above, where we specify that the side boundary conditions exactly match those of the parabolic Toda solution above. The setup here is amenable to a change of measure, where we remove the parabola and shift all endpoints to sit at the same location. The cost of doing this is that we change the Hamiltonian $e^x$ to the line-index dependent Hamiltonian $i(e^x - x - 1)$. Since this Hamiltonian increases steeply away from $0$, this suggests we have strong concentration around a flat solution. We establish this concentration crudely, using not much more than a bound on the partition function.
	
	It remains to relate the parabolic boundary problem back to the KPZ line ensembles. We do this in Section \ref{sec:bounding-final}. The basic idea here should feel familiar to researchers using Brownian Gibbs arguments (in particular, it is reminiscent of Chapter 2 of \cite{aggarwal2023}): in essence, we use monotone couplings for $\Ham$-Brownian line ensembles, raising and lowering boundary conditions appropriately and increasing or decreasing parabolic apertures to gain comparisons with different parabolic boundary problems. Our initial inputs to help construct these monotone couplings are the exponential moment control from Theorem \ref{thm:e_moments}, and a strong modulus of continuity for the KPZ line ensemble from \cite{wu2025optimaltransport}. With these initial inputs, we can recursively control line locations at $0$, and suprema and infima of lines to give an inequality chain which will eventually bound the moderate upper and lower tails of line locations in terms of their deep upper tails. Plugging in the bound on the deep upper tail from Remark \ref{R:small-n} (which itself is just Theorem \ref{thm:e_moments} and Markov's inequality) then gives the result.

	\subsection{Open problems}
	
	This paper provides only a first step toward understanding the overall distribution of the KPZ line ensemble. Indeed, we expect that there is a lot of room to improve the estimates in the present paper and bring in new tools (i.e., we have not used gap monotonicity), revealing new qualitative properties of the KPZ line ensemble. First, we expect that not only does the KPZ line ensemble eventually drift to $+\infty$ with the line index, but it is eventually totally ordered in a rigid way that closely matches the spacing in Theorem \ref{thm:e_moments}.
	
	\begin{con}
		Fix $t > 0$. Then almost surely as $n \to \infty$,
		\begin{align*}
			\mathcal{H}^{(t)}_{n +  1}(0) - \mathcal{H}^{(t)}_n (0) = \log(n/t) + o (1)
		\end{align*}
	\end{con}
	The reason to expect the $o(1)$ error here can be seen by considering the following $\Ham$-Brownian resampling problem with three lines. Let $n \gg t$, and let $\mathcal{H}^{(t)}_{n +  1}, \mathcal{H}^{(t)}_{n-1}$ be fixed at locations $e_{n+1}^{(t)} - x^2/(2t), e_{n-1}^{(t)} - x^2/(2t)$ on a long interval. Since $\mathcal{H}^{(t)}_{n +  1} - \mathcal{H}^{(t)}_{n-1} = \log(n(n-1)/t^2) \gg 0$, if we resample the line $\mathcal{H}^{(t)}_n$ according to the weighting in \eqref{E:RN-intro}, a balance of costs argument suggests that the line $\mathcal{H}^{(t)}_n$ should stay within $O(n^{-1/4 + \epsilon})$ of the Toda location $e_{n}^{(t)} - x^2/(2t)$. Continuing this line of reasoning, it is reasonable to expect that lines themselves also become rigid in the large $n$ limit.
	
	\begin{con}
		Fix $t > 0$. Then as $n \to \infty$, the sequence of functions
	$
	\mathcal{H}^{(t)}_n - \mathcal{H}^{(t)}_n(0)
$
converges almost surely, uniformly on compact sets to the deterministic parabola $-x^2/(2t)$.
	\end{con}
We are less confident on how tightly the point process $\mathcal{H}^{(t)}_n(0)$ itself concentrates around the Toda locations $e_{n}^{(t)} - x^2/(2t)$. It does not seem unreasonable to expect a lower order deterministic correction here, or even a random term which survives as $n \to \infty$. 

We end by conjecturing an analogue of the celebrated Airy line ensemble characterization from \cite{aggarwal2023}. 
	
	\begin{con}
		\label{C:airy-con}
		Let $\Lag = \{ \Lag_n \}_{n = 1}^{\infty}$ be a collection of random curves with the $\Ham$-Brownian Gibbs property with $\Ham(x) = e^x$. Suppose for any $\epsilon > 0$, there exists a constant $\kappa(\epsilon) > 0$ with 
		\begin{align*}
			\prob \big( |\Lag_1(x) - \frac{x^2}{2t}| \leq \epsilon x^2 + \kappa(\epsilon) \big) \geq 1 - \epsilon
		\end{align*}
		for all $x \in \R$.
		Then $\Lag$ is equal in distribution to the KPZ$_t$ line ensemble, up to an independent affine shift. 
	\end{con}

While there are fewer known applications for Conjecture \ref{C:airy-con} when compared to the Airy characterization in \cite{aggarwal2023}, we expect that it may be much easier to establish. For example, arguably the most difficult piece of \cite{aggarwal2023} is a global law, establishing Airy point locations at a law of large numbers level for non-intersecting line ensembles with a parabolic first curve. We believe that the methods in the present paper could give the same result in the positive temperature setting, with modifications required in order to eliminate our reliance on the modulus of continuity for the KPZ line ensemble from \cite{wu2025optimaltransport}.

\subsubsection*{Conventions Throughout the Paper}
We often suppress the dependence of objects on $t$ or other subscripts and superscripts, e.g., we will often write $\mathcal{H}_n$ instead of $\mathcal{H}_n^{(t)}$. We write $\prob_{\mathcal{F}}(A) \defeq \E [ \mathbbm{1}_A | \mathcal{F}]$ for the probability of the event $A$ conditioned on a $\sigma$-algebra $\mathcal{F}$.  Next, we use the shorthand $[a, b]_{\Z} \defeq [a, b] \cap \Z$. Finally, our constants may change from line to line within a proof.  

	\section{Preliminaries}
	
	In this section, we define $\mathbf{H}$-Brownian line ensembles and introduce the KPZ line ensemble as the edge scaling limit of the O'Connell-Yor polymer. Additionally, we define the multi-layer stochastic heat equation and clarify its precise relation to the KPZ line ensemble.

	\subsection{$\mathbf{H}$-Brownian Gibbs property and line ensembles} 
	\label{S:gibbs-section}
	\begin{defn} 
		Consider intervals $\Sigma \subset \N$ and $\Lambda \subset \R$. Let $C(\Sigma \times \Lambda, \R)$ be the set of continuous functions from $\Sigma \times \Lambda \to \R$ endowed with the topology of uniform convergence on compact sets. For $f \in C(\Sigma \times \Lambda, \R)$, we write $f_i(x) = f(i, x)$ and view $f$ as a sequences of functions $f_i, i \in \Sigma$.
		
		A \textbf{$(\Sigma \times \Lambda)$-indexed line ensemble $\Lag$} is a random variable taking values in $C(\Sigma \times \Lambda, \R)$. 
	\end{defn}
	
	Now fix $k_1 < k_2 \in \N, [a, b] \subset \R$ and two vectors $\vec{x}, \vec{y} \in \R^{k_2 - k_1 + 1}$. We define the \textbf{free Brownian bridge measure} $\prob_{\textrm{free}}^{k_1, k_2, [a, b], \vec{x}, \vec{y}}$ be the law of $k_2 - k_1 +1$ independent Brownian bridges $B = (B_{k_1}, \dots, B_{k_2}) 
	\in  C(\{ k_1 , \ldots , k_2\} \times [a, b])$ with $B_i(a) = x_i, B_i(b) = y_i$. We simply write $\prob_{\textrm{free}}$ (and similarly $\E_{\textrm{free}}$ when the choices of $k_1, k_2, [a, b], \vec{x}, \vec{y}$ are clear from context. $\Ham$-Brownian line ensembles are constructed by tilting the free Brownian bridge measure through an interaction Hamiltonian $\mathbf{H}$. 
    
	\begin{defn} \label{defn:HBMelon_defn}
		Consider a measurable function $\mathbf{H}: [-\infty, \infty) \to [0,  \infty]$ with $\mathbf{H}(-\infty) = 0$, referred to as the \textbf{Hamiltonian}, together with lower index and upper index boundary conditions $f, g:[a, b] \to \R \cup \{\pm \infty\}$. For $h \in C(\{ k_1 , \ldots , k_2\} \times [a, b])$ define its \textbf{Boltzmann weight} with respect to $\Ham, f, g$ by
		\begin{align*}
			W_{\mathbf{H}}^{f, g} (h) \defeq \exp \left(- \sum_{i = k_1-1}^{k_2} \int_{a}^{b} \mathbf{H} \circ \Delta_i h (s) \, ds  \right), \qquad \text{ where } \qquad \Delta_i h = h_{i+1} - h_i.
		\end{align*}
		where $h_{k_1 - 1} = f, h_{k_2 + 1} = g$. 
		We define the \textbf{$\Ham$-Brownian bridge measure} $\prob_{\mathbf{H}}^{f, g} = \prob_{\mathbf{H}}^{k_1, k_2, [a, b], \vec{x}, \vec{y}, f, g}$ on $C(\{ k_1 , \ldots , k_2\} \times [a, b])$ by the Radon-Nikodym derivative formula
		\begin{align*}
			\frac{d \prob_{\mathbf{H}}^{f, g} }{d \prob_{\textrm{free}}} (h) \defeq \frac{1}{Z^{f, g}_\Ham} W_{\mathbf{H}}^{f, g} (h), 
		\end{align*}
		where $Z^{f, g}_\Ham = Z_{\mathbf{H}}^{k_1, k_2, (a, b), \vec{x}, \vec{y}, f, g}$ is the partition function
		\begin{align}
			Z_{\mathbf{H}}^{f, g} \defeq \E_{\textrm{free}} [W_{\Ham}^{f, g} (B)].
		\end{align}
		We say $\Lag$ is an \textbf{$\Ham$-Brownian line ensemble} if $\Lag \sim \prob_\Ham^{f, g}$.
	\end{defn}
	
	For the purposes of our analysis, it is useful to extend the definition of an $\Ham$-Brownian line ensemble to allow for a vector of Hamiltonians, $\vec{\Ham} = (\Ham_{k_{1} - 1} , \ldots , \Ham_{k_2})$. All parts of our definition are kept the same except we simply redefine the Boltzmann weight as:
	\begin{align}
		W_{\mathbf{H}}^{f, g} (h) \defeq \exp \left(- \sum_{i = k_1-1}^{k_2} \int_{a}^{b} \mathbf{H}_i \circ \Delta_i h(s) \, ds  \right).
	\end{align}
	We are now ready to define the $\Ham$-Brownian Gibbs property. 
	\begin{defn} Consider a Hamiltonian $\Ham$, intervals $\Sigma \subset \N$ and $\Lambda \subset \R$, and a $(\Sigma \times \Lambda)$-indexed line ensemble $\Lag$. We say $\Lag$ satisfies the $\Ham$-Brownian Gibbs property if the following property holds for every $k_1 \le k_2 \in \Sigma$ and $[a, b] \subset \Lambda$.
		
		Set $\vec{x} = (\Lag_{k_1}(a) , \ldots ,\Lag_{k_2}(a))$,  $\vec{y} = (\Lag_{k_1}(b) , \ldots , \Lag_{k_2}(b))$ and set $f=  \Lag_{k_1 - 1}$ and $g = \Lag_{k_2 + 1}$ (if $k_1 - 1 \not\in \Sigma$, then let $f = + \infty$ and if $k_2 + 1 \not\in \Sigma$, let $g = - \infty$; since $\Ham(-\infty) = 0$ this amounts to dropping the term involving $f$ or $g$ in the Boltzmann weight). Define the exterior $\sigma$-algebra $\mathcal{F}_{\textrm{ext}}$ by:
		\begin{align}
			\mathcal{F}_{\textrm{ext}} \defeq \sigma \big\{    \Lag \big|_{\Sigma \times \Lambda \setminus (\{k_1, \dots, k_2\} \times [a, b])}  \big\}
		\end{align}
		Then the conditional law of $\Lag|_{[k_1, k_2]_{\Z} \times (a, b)}$ given $\mathcal{F}_{\textrm{ext}}$ is $\prob_{\Ham}^{k_1, k_2, [a, b], \vec{x}, \vec{y}, f, g}$.

	\end{defn}
	
	It is straightforward to check that for a compact domain $\Sigma \times \Lambda$, a $(\Sigma \times \Lambda)$-indexed line ensemble satisfies the $\Ham$-Brownian Gibbs property if and only if it is a $\Ham$-Brownian line ensemble. However, when we move to non-compact domains, it is no longer possible to directly define a notion of a $\Ham$-Brownian line ensemble through a Radon-Nikodym derivative. However, we can still define line ensembles with the $\Ham$-Brownian Gibbs property, and this property is extremely useful for analysis of these objects.
	
	We note in passing that the non-intersecting Brownian Gibbs property for the Airy line ensemble can be defined through the Hamiltonian $\Ham_{\textrm{ord}} (x) = \infty \mathbf{1}(x \geq 0)$.
	
	\subsection{The O'Connell-Yor and KPZ line ensembles} \label{subsec:KPZ_Line_Ensemble_defn}
	
	The KPZ line ensemble is the scaling limit of a line ensemble connected to the O'Connell-Yor polymer. As there is some inconsistency regarding the definition of the KPZ line ensemble in the literature, we precisely define both the O'Connell-Yor polymer and the scaling limit to the KPZ line ensemble here. 
	
	\begin{defn}
		For $a \le b \in \R$ and $n \le m \in \Z$ an \textbf{up-right path} $\phi$ from $(a, n)$ to $(b, m)$ is a non-decreasing, cadlag, function from $[a, b]$ to $[n, m]_\Z$. Paths from $(a, n)$ to $(b,m)$ are in bijection with sequences $a = s_n \le s_{n+1} \le \cdots \le s_{m-1} \le b$, defined by setting $s_i = \inf \{s \in [a, b] : \phi(s) > i\}$. We call the times $s_i$ \textbf{jump times} for the path $\phi$. We say that a $k$-tuple of paths $\phi = (\phi_1, \dots, \phi_n)$ from $(a, n)$ to $(b,m)$ is \textbf{non-intersecting} if $\phi_1 > \phi_2 > \cdots > \phi_n$.

		Next, let $B = \{ B_i \}_{i = 1}^{\infty}$ be an infinite sequence of independent standard Brownian motions, let $\phi$ be a path from $(0, 1)$ to $(s, N)$ for some $s \ge 0, N \in \N$ and define the energy of $\phi$ by:
		\begin{align}
			E(\phi) &= B_1(s_1) + (B_2(s_2) - B_2(s_1)) + \cdots + (B_{N}(s) - B_{N}(s_{N-1})).
		\end{align}
		For $n \in \{1, \dots, N\}$ and $t > 0$ define the $n$-path \textbf{O'Connell-Yor polymer partition function}
		\begin{align}
			Z^{N}_{n}(t) &= \int_{D_{n, t}} e^{\sum_{i = 1}^n E(\phi_i)} d \mu(\phi).
		\end{align}
		Here $D_{n, t}$ is the space of non-intersecting $n$-tuples of paths $\phi$ from $(0, 1)$ to $(t, N)$. Using the correspondence between paths and jump times, we can view $D_{n, t}$ as a subset of $(\R^N)^n$. The measure $\mu$ is simply Lebesgue measure on $(\R^N)^n$. Next, define the \textbf{O'Connell-Yor line ensemble} $X^N$ by
		\begin{align}
			X^{N}_{n}(s) = \log \big(  \frac{Z^{N}_n(s)}{Z^{N}_{n - 1} (s)} \big), \qquad (n, s) \in [1, N]_{\Z} \times (0, \infty).
		\end{align} 
		with the convention that $Z^N_{0} = 1$. 
	\end{defn}	
	
	O'Connell \cite{o2012directed} demonstrated that the process $X^n$ has a description as a Doob$-h$ transform, and identified the generator for the process. From this description, Corwin and Hammond \cite{corwin2016kpz} demonstrated that the O'Connell-Yor line ensemble possesses an $\mathbf{H}$-Brownian Gibbs property (this can also be extracted from the earlier work of Katori \cite{katori2011connell, katori2012survival}).
	
	\begin{lemma}[Proposition 3.4, \cite{corwin2016kpz}]
		The O'Connell-Yor line ensemble $X^n$ possesses an $\Ham$-Brownian Gibbs property with the Hamiltonian $\Ham(x) = e^x$. 
	\end{lemma}
	
	We define the KPZ line ensemble as a scaling limit of the O'Connell-Yor line ensemble.
	
	\begin{thm}
		For $N \in \N, t > 0,$ and $x \in \R$ let
		\begin{align}
			C(N, t, x) \defeq (t^{1/2} N^{-1/2})^{N-1} \exp \Big(  N + \frac{\sqrt{tN} +x}{2} + xt^{-1/2} N^{1/2}\Big)
		\end{align}
		and define
		\begin{align}
			\mathcal{H}^{t, N}_{n} (x) \defeq X^{N}_n (\sqrt{tN} + x) - \log \big( C(N , t, x) \big)
		\end{align}
		Then as $N \to \infty$, the process $\mathcal{H}^{t, N}$ converges in distribution to an $(\N\times \R)$-indexed line ensemble $\mathcal H^{(t)}$, where the underlying topology is uniform-on-compact convergence on subsets of $\N \times \R$. The process $\mathcal H^{(t)}$ is the \textbf{KPZ$_t$ line ensemble}, and as with the O'Connell-Yor line ensembles, $\mathcal H^{(t)}$  possesses an $\Ham$-Brownian Gibbs property with the Hamiltonian $\Ham(x) = e^x$. Finally, the process $(x, n) \mapsto \mathcal H^{(t)}_n(x) + x^2/(2t)$ is stationary in $x$.
	\end{thm}
	
	The tightness of $\mathcal{H}^{t, N}$ and the limiting Gibbs property (which is an easy consequence of tightness) was shown in \cite{corwin2016kpz}. Uniqueness of the limit point and the resulting stationarity was proven by Nica \cite{Nica_2021}.

	\begin{remark}
		Note that our definition of the KPZ line ensemble as the scaling limit of the line ensembles $\mathcal{H}^{t, N}$ coincides with that of \cite{Nica_2021} and the originally published version of \cite{corwin2016kpz}. However, in the most recent arXiv version of \cite{corwin2016kpz}, the KPZ line ensemble was redefined as the scaling limit $\tilde {\mathcal H}^t$ of the line ensembles $\tilde{\mathcal{H}}^{t, N}$, given
		\begin{align} \label{modified_ensemble}
			\tilde{\mathcal{H}}^{t, N}_n(x) = \mathcal{H}^{t, N}_n(x) - \log \big( t^{1 - n} (n - 1)! \big).
		\end{align}
		The scaling limits differ by the line index dependent shift $\log \big(t^{1 - n} (n - 1)! \big)$. The advantage of working with $\tilde {\mathcal H}^t$ is that it yields a simpler connection to the multi-layer stochastic heat equation due to \eqref{E:Htn}. However, introducing the line-index-dependent shift does not preserve the $\mathbf{H}$-Brownian Gibbs property with $\mathbf{H}(x) = e^x$ and instead results in a line ensemble with a Gibbs property with a vector of Hamiltonians of changing strength from line to line.
	\end{remark}

	\subsection{Key Inputs}
		
	\subsubsection{Monotonicity for $\Ham$-Brownian line ensembles}
	
	The first two useful lemmas are from \cite{corwin2016kpz} and concern stochastic monotonicity of $\Ham$-Brownian line ensembles with convex Hamiltonians $\Ham$. 
	The case of vector Hamiltonians $\vec{\Ham}$ is not proven in \cite{corwin2016kpz}, but the lemmas still hold in the vector Hamiltonian case and the proofs are almost identical. For completeness, we include proofs for vector-valued Hamiltonians in the appendix. 
	
	\begin{lemma}[Non-vector Hamiltonian case given in Lemma 2.6 of \cite{corwin2016kpz}] \label{lemma:stoch_dom1}
		Fix $k_1 < k_2 \in \N$, $a < b \in \R$, and two vectors $\vec{x}, \vec{y} \in \R^{k_2 - k_1 + 1}$. Fix either a convex Hamiltonian $\Ham$, or a vector of convex Hamiltonians $\Ham = (\Ham_{k_1 -1}, \ldots , \Ham_{k_2})$. Fix two pairs of measurable functions $f^{i} , g^{i}:[a, b] \to \R \cup \{\pm \infty\}$ for $i = 1, 2$ such that $f^{1} \leq f^{2}$ and $g^{1} \leq g^{2}$. 
		
		Let $\Lag^{(i)}$ for $i = 1,2$ be two $\Ham$-Brownian line ensembles distributed according to $\prob_{\Ham}^{(i)} \defeq \prob_{\Ham}^{\vec{x}, \vec{y}, [a, b], f^{i}, g^{i}}$. Then there exists a coupling of $\Lag^{(1)}, \Lag^{(2)}$ so that $\Lag^{(1)}_{j}(x) \leq \Lag^{(2)}_j(x)$ for all $j \in [k_1, k_2]_{\Z}$, $x \in [a, b]$.
	\end{lemma}
	
	\begin{lemma}[Non-vector Hamiltonian case given in Lemma 2.7 of \cite{corwin2016kpz}]  \label{lemma:stoch_dom2}
		Fix $k_1 < k_2 \in \N$, $a < b \in \R$, and a pair of measurable functions $f, g:[a, b] \to \R \cup \{\pm \infty\}$. Fix either a convex Hamiltonian $\Ham$, or a vector of convex Hamiltonians $\Ham = (\Ham_{k_1 -1}, \ldots , \Ham_{k_2})$. Fix two pairs of vectors $\vec{x}^{(i)}, \vec{y}^{(i)} \in \R^{k_2 - k_1 + 1}$ for $i = 1, 2$ where $x^{(1)}_j \leq x^{(2)}_j$ and $y^{(1)}_j \leq y^{(2)}_j$.
		
		Let $\Lag^{(i)}$ for $i = 1,2$ be two $\Ham$-Brownian line ensembles distributed according to $\prob_{\Ham}^{(i)} \defeq \prob_{\Ham}^{\vec{x}^{(i)}, \vec{y}^{(i)}, [a, b], f, g}$. Then there exists a coupling of $\Lag^{(1)}, \Lag^{(2)}$ so that $\Lag^{(1)}_{j}(x) \leq \Lag^{(2)}_j(x)$ for all $j \in [k_1, k_2]_{\Z}$, $x \in [a, b]$.
	\end{lemma}

	\subsubsection{Properties of the KPZ line ensemble}
	
	The follow result concerns the single point upper and lower tails of the KPZ Equation, or equivalently the top curve of the KPZ line ensemble.
	
	\begin{lemma}[Theorem 1.1 of \cite{Corwin_2020_lowertail}, Theorem 1.11 of \cite{corwin2020kpzequationtailsgeneral}] \label{lemma:kpzeqn_tails} Fix $t_0> 0$. Then there exists $d > 0$ depending only on $t_0$ such that for all $t > t_0$ and $m > 0$:
		\begin{align*}
			\prob_{\Ham} (\mathcal H^{(t)}(0) + \frac{t}{24} \leq -m t^{1/3}) & \leq 2 \exp \big(   -d m^{5/2} \big), \\
			\prob_{\Ham} (\mathcal H^{(t)}(0) + \frac{t}{24} \geq m t^{1/3}) & \leq 2 \exp \big(   -d m^{3/2} \big).
		\end{align*}
	\end{lemma}
	Note that in Theorem 1.1 of \cite{Corwin_2020_lowertail}, a stronger and much more precise lower tail bound is obtained, This was later optimized in 
\cite{tsai2022exact}. For our purposes, the crude bound above suffices.
	
	We next include a novel result of \cite{wu2025optimaltransport} which uses Cafarelli's contraction theorem from optimal transport theory to control the regularity of the lines in the KPZ line ensemble. The result in \cite{wu2025optimaltransport} applies to more general line ensembles such as $\beta$-Dyson Brownian motion and other $\Ham$-Gibbsian line ensembles, but we only include the result for the KPZ line ensemble.
	
	\begin{lemma}[Theorem 2.8 and Proposition 3.4 in \cite{wu2025optimaltransport}] \label{lemma:opt_tran} 
		Fix $t > 0$, let $\mathcal H^{(t)}$ be the KPZ line ensemble with parameter $t$. Then for all $t, n,$ and $x \in \R$, the expectation $\E |\mathcal H^{(t)}_n(x)| < \infty$ and $\E \mathcal H^{(t)}_n(x) = \E \mathcal H^{(t)}_n(0) - \tfrac{x^2}{2t}$. Moreover, if we let $\hat{\mathcal{H}}^t_n(x) = \mathcal H^{(t)}_n(x)+ \tfrac{x^2}{2t}$, then there exists a universal constants $d > 0$ such that for all $[a, b] \subset \R$, $j \in \N$ and $m > 0$ we have:
		\begin{align*}
			\prob \Big( \sup_{x,y \in [a, b] , \, t \neq s}   \frac{|\hat{\mathcal{H}}^t_j(x) - \hat{\mathcal{H}}^t_j(y) |}{\sqrt{|x - y| \log (2 (b - a)/(x- y))}}  \geq m \Big) \leq 2 e^{- dm^2}.
		\end{align*}
	\end{lemma}
	
	\subsubsection{One more fact}
	
	We finish by recording a standard bound on the probability of a Brownian bridge staying in a tube.
	\begin{lemma} \label{lemma_tube} Let $L > 0$ and let $B: [0, L] \to \R$ be a Brownian bridge with $B(0) = B(L) = 0$. Then for all $r > 0$ and $d > \pi^2/8$, there exists a constant $c_d > 0$ such that
		\begin{align*}
			\prob \Big( \sup_{t \in [0, L]} |B_t| \leq r \Big) \geq c_d \exp(- d L/r^2).
		\end{align*}
	\end{lemma}
	
	\begin{proof}
		By Brownian scaling, it suffices to prove the result when $r = 1$. Next, let $W_t$ be a standard Brownian motion on $[0, \infty)$ and let $T = \min \{t \ge 0: |W_t| = 1\}$. Now, $|W_t|$ is a Brownian motion, reflected off of the line $x = 0$, whereas $|B_t|$ is a similarly reflected Brownian motion, conditioned on $B_L = 0$. Therefore we can couple $W_t, B_t$ so that $|B_t| \le |W_t|$ for all $t \in [0, L]$, and so
		\begin{equation}
			\label{E:BM-compare}
			\prob \Big( \sup_{t \in [0, L]} |B_t| \leq 1 \Big) \ge \prob( T \ge L).
		\end{equation}
		Next, fixing $\theta \in (0, \pi^2/8)$, and using that $t \mapsto \cos(\sqrt{2 \theta} W_t) \exp(\theta t)$ is a martingale, we have that
		\begin{align*}
			1= \E \cos(\sqrt{2 \theta} W_{t \wedge T}) \exp(\theta (t \wedge T)) &\ge \cos(\sqrt{2 \theta}) \E \exp(\theta (t \wedge T)),
		\end{align*}
		which upon taking $t \to \infty$ implies that $\E \exp(\theta T) \le 1/\cos(\sqrt{2 \theta})$. This exponential bound for all $\theta < \pi^2/8$ gives us the required integrability to take $T \to \infty$ in the equality above for any fixed $\theta < \pi^2/8$ to get that $\E\exp (\theta T) = 1/\cos(\sqrt{2 \theta})$. The result follows since $\cos(\sqrt{2 \theta}) \to 0$ as $\theta \to \pi^2/8$.
	\end{proof}

	\section{Exponential moment of gaps in the KPZ line ensemble}
	
	In this section, we prove Theorem \ref{thm:e_moments}. We first establish an integration-by-parts formula for more general $\Ham$-Brownian line ensembles.
	For this section we will fix any continuous Hamiltonian $\Ham: \R \cup \{-\infty\} \to [0, \infty)$. In the remaining sections, we specialize to the choice $\Ham(x) = e^x$.

	\begin{thm} \label{thm:int_parts} Let $\mathbf{H}$ be any continuously differentiable Hamiltonian such that $|\Ham'(x)| = \exp(o(x^2))$ as $x \to \pm \infty$. Let $\Lag$
		be an $\mathbf{H}$-Brownian line ensemble with domain $[k_1, k_2]_\Z \times [-T, T]$, and law $\prob_{\Ham} = \prob_{\Ham}^{k_1, k_2, [a, b], \vec{x}, \vec{y}, f, g}$ for vectors $\vec{x}, \vec{y} \in \R^{k_2 - k_1 + 1}$ and boundary conditions $f, g:[a, b] \to \R$.
		
		Let $a < b \in [-T, T]$ and $i \in [k_1, k_2]_\Z$, define the random variable $Y = \Lag_i(b) - \Lag_i(a)$, and let $\mu_{a, b} = (\Lag_i(T) - \Lag_i(-T))(b - a)/(2T)$. Let $F: \R \to \R$ be any continuously differentiable function with bounded derivative. Then
		\begin{align*}
			&\E_{\Ham} \big[ (Y - \mu_{a, b})F(Y) \big] = \E_{\Ham} \big[ F'(Y)  \big] \big( (b - a) - \frac{(b - a)^2}{2T} \big)  \\
			& +  \int_{-T}^{T} \E_{\Ham}  [F(Y) (\Ham' (\Delta_i\Lag(s))  - \Ham'(\Delta_{i-1}\Lag(s)))]  \operatorname{Cov}^{T, a, b}(s) \, ds ,
		\end{align*}
		where
		\begin{align}
			\operatorname{Cov}^{T, a, b} (s) \defeq  \frac{(b \wedge s + T) (T - b \vee s)}{2 T }  - \frac{ (a \wedge s + T ) (T - a \vee s)  }{2 T} .
		\end{align}
	\end{thm}
	\begin{remark}If we specialize to the KPZ line ensemble with Hamiltonian $\Ham(x) = e^{x}$, then, given the additional assumption that $\mathcal H^{(t)}_{i+1} - \mathcal H^{(t)}_{i} \xrightarrow[t \to \infty]{} -\infty$ for all $i$, Theorem \ref{thm:int_parts} implies that for any index $i$, conditional on $\mathcal H^{(t)}|_{(\{i\} \times [a, a + 2b])^c}$ the increment $\mathcal H^{(t)}_i(a + b) - \mathcal H^{(t)}_i(a)$ converges to a normal random variable with mean $[\mathcal H^{(t)}(a + 2b) - \mathcal H^{(t)}(a)]/2$ and variance $b/2$. When combined with the stationarity of $\mathcal H^{(t)}(x) + x^2/(2t)$, this implies that the process $\mathcal{H}^{(t)} - \mathcal{H}^{(t)}(0)$ converges as $t \to \infty$ to a line ensemble of infinitely many independent Brownian motions. This identifies one joint stationary law of the multi-layer stochastic heat equation. A family of these joint stationary laws can be constructed by adding in a slope.
	\end{remark}    
	
	To prove Theorem \ref{thm:int_parts} we will introduce an approximation of the $\Ham$-Brownian line ensemble which uses a discretization of the Boltzmann weight. More precisely, let $N \in \N$, $T > 0$, let $\delta = \delta(N, T) = \frac{2T}{2^N}$ and define the mesh:
	\begin{align*}
		\Pi_N &= \{ -T, -T + \delta, - T + 2 \delta , \ldots, T - \delta, T\} = \{ t_0 = -T < t_1 < \cdots < t_{2^N} = T \}.
	\end{align*}
	Given $k_1, k_2, [a, b], \vec{x}, \vec{y}, f$, and $g$, for $h \in C([k_1, k_2]_\Z \times [a, b])$ define the discretized Boltzmann weight
	\begin{equation}
		\label{E:disc-boltzmann}
		W_{\mathbf{H}_N}(h) \defeq \exp \left(- \delta \sum_{i = k_1-1}^{k_2} \sum_{j=1}^{2^N} \mathbf{H} \circ \Delta_i h (t_j) \right).
	\end{equation}
	We define the partition function $Z_{\Ham^N}$ and the measure $\prob_{\Ham^N}$ exactly as before, but with the discretized Boltzmann weight $W_{\mathbf{H}_N}$ in place of $W_{\mathbf{H}}^{f, g}$. Observe that as $N \to \infty$, we have
	\begin{equation}
		\label{E:weak-cvg}
		W_{\mathbf{H}_N} \to W_{\mathbf{H}} = W_{\mathbf{H}}^{f, g}, \qquad Z_N := \E_{\textrm{free}} W_{\mathbf{H}_N} (\Lag) \to Z = \E_{\textrm{free}} W_{\mathbf{H}} (\Lag),
	\end{equation}
	where the first convergence above is pointwise. Indeed, the first convergence follows from the continuity of $\Ham$, which ensures convergence of the Riemann sum in \eqref{E:disc-boltzmann} to an integral, and the second convergence is simply the bounded convergence theorem, using that all Boltzmann weights are bounded above by $1$. Together these facts imply weak convergence of $\prob_{\Ham^N}$ to $\prob_{\Ham}$, justifying our discretization.

	\begin{proof}[Proof of Theorem \ref{thm:int_parts}] First, it suffices to prove that the equality holds conditionally on the lines $\Lag_{i-1}, \Lag_{i+1}$, as then the unconditional equality can be recovered by taking an expectation. By the $\Ham$-Brownian Gibbs property (and re-indexing), this reduces the theorem to the case of a single path $\Lag_1$ with deterministic upper and lower boundary conditions $f = \Lag_0, g = \Lag_2$. Moreover, it is enough to prove the equality for $a, b \in \bigcup_{N\in \N} \Pi_N$, as the general result then follows by rational approximation, as all terms are continuous functions of $a, b$ (this uses the tail bound on $|\Ham'|$ to justify continuity of the integral in $a, b$). In this setting, we first prove an approximate version of this equality for the measures $\prob_{\Ham_N} = \prob_{\Ham_N}^{f, g}$, where $N$ is large enough so that $a < b \in \Pi_N$. We have
		\begin{align} \label{int_parts_init}
			\E_{\Ham_N} \big[  (Y- \mu_{a, b}) F(Y) \big] = Z_{N}^{-1} \E_{\textrm{free}} \big[\bar \Lag_1(b) F(Y) W_{\mathbf{H}_N}(\Lag)  \big] - Z_{N}^{-1} \E_{\textrm{free}} \big[\bar \Lag_1(a) F(Y)  W_{\mathbf{H}_N}(\Lag)  \big],
		\end{align}
		where $\bar \Lag(x) = \Lag(x) - \E_{\textrm{free}} \Lag_1(x)$. 
		Under the measure $\prob_{\textrm{free}}$, the discretized path $(\Lag_1(x) : x \in \Pi_N \setminus \{-T, T\})$ is a centered Gaussian random vector with covariance
		\begin{align} \label{bbridge_cov}
			\E \bar \Lag_1(x)\bar \Lag_1(y) = \frac{(T + x \wedge y) (T - x \vee y)}{2 T}.
		\end{align}
		Using Gaussian integration by parts, the first term on the right-hand side of \eqref{int_parts_init} is equal to
		\begin{align*}
			Z_N^{-1}\sum_{k = 1}^{2^N - 1} \E_{\textrm{free}} \big[ \bar \Lag_1(b) \bar \Lag_1(t_k) \big] \E_{\textrm{free}} \big[ \frac{\partial }{\partial \Lag_{i}(t_k)}  F \big(\Lag_i(b) - \Lag_i(a) \big)  \exp \big( - \sum_{i = 0}^1 \sum_{j = 0}^{2^N} \delta \Ham (\Delta_i \Lag(t_j)) \big) \big] .
		\end{align*}
		After inputting the covariances from \eqref{bbridge_cov} and computing partial derivatives, this becomes
		\begin{align*}
			&\sum_{k = 1}^{2^N - 1} \big( \frac{(T + b \wedge t_k) (T - b \vee t_k)}{2 T}   \big)  \cdot \delta \cdot Z_N^{-1}\E_{\textrm{free}} \big[ F(Y)  \big( \Ham'(\Delta_1 \Lag(t_k))  - \Ham'(\Delta_0 \Lag(t_k)) W_{\Ham_N}(\Lag)  \big] \\
			& + \frac{(T + b) (T - b)}{2T} Z_{N}^{-1} \cdot \E_{\textrm{free}} \big[ F'(Y) W_{\Ham_N}(\Lag) \big] - \frac{(T + a) (T - b)}{2T} \cdot Z_N^{-1}\E \big[ F'(Y) W_{\Ham_N}(\Lag)  \big].
		\end{align*}
		Now, upon taking $N \to \infty$, the last two terms above converge to the same objects with $Z_N \mapsto Z, W_{\Ham_N} \mapsto W_{\Ham}$. This uses \eqref{E:weak-cvg}. To deal with the Riemann sum above, observe that by taking the expectation outside, we can write this as $\E_{\textrm{free}} [\phi_n^b(\Lag)]$ for a sequence of functions $\phi_n^b$ converging pointwise to the limit
		$$
		\phi^b(\Lag) = Z^{-1} W_\Ham (\Lag) F(Y) \int_{-T}^{T}  \frac{(T + b \wedge s) (T - b \vee s)}{2T} (\Ham' \circ \Delta_1 \Lag - \Ham' \circ \Delta_0 \Lag)(s) ds.
		$$
		We can extend this pointwise convergence to convergence in expectation, $\E_{\textrm{free}} [\phi_n^b(\Lag)] \to \E_{\textrm{free}} [\phi^b(\Lag)]$ by appealing to the tail bound $|\Ham'(x)| = \exp(o(x^2))$ as $x \to \pm \infty$, which implies uniform integrability of the functions $\phi_n$ since the difference $\Delta_i \Lag$ has uniform Gaussian tails. In summary, as $N \to \infty$, $Z_{N}^{-1} \E_{\textrm{free}} \big[\bar \Lag_1(b) F(Y) W_{\mathbf{H}_N}(\Lag)  \big]$ converges to
		\begin{align*} 
			\E_{\textrm{free}} \phi^b(\Lag) + Z^{-1} \E_{\textrm{free}} [F'(Y) W_\Ham(\Lag)] \big( \frac{(T - b) (b - a)}{2T} \big).
		\end{align*}
		An identical argument for the second term of \eqref{int_parts_init} gives
		\begin{align*}
			\lim_{N\to \infty}	Z_{N}^{-1} \E_{\textrm{free}} \big[\bar \Lag_1(a) F(Y)  W_{\mathbf{H}_N}(\Lag)  \big]  = \E_{\textrm{free}} \phi^a(\Lag)- \E_{\textrm{free}} [F'(Y) W_\Ham(\Lag)] \big(  \frac{(T + a) (b - a)}{2T}\big).
		\end{align*}
		Subtracting the two expressions and rewriting the expectations in terms of $\E_\Ham$ concludes the proof. 
	\end{proof} 
	
	\begin{proof}[Proof of Theorem \ref{thm:e_moments}]
		In the proof, let $\mathcal H = \mathcal H^{(t)}$ be the KPZ$_t$ line ensemble. We will prove the theorem inductively, by showing that for all $n \in \N$, the exponential moments $\E e^{\mathcal{H}_{n + 1}(0) - \mathcal{H}_n(0)}, \E e^{\mathcal{H}_n(0) - \mathcal{H}_{n-1}(0)}$ both exist and we have the formula
		\begin{align}
			\label{E:exponential-gap}
			\E e^{\mathcal{H}_{n + 1}(0) - \mathcal{H}_n(0)} - \E e^{\mathcal{H}_{n}(0) - \mathcal{H}_{n - 1}(0)}  &= \frac{1}{t},
		\end{align}
		where $\mathcal H_0 = \infty$. We prove the base case and the inductive step together, pointing out the minor differences as we go. Fix $n \in \N$ and (for the inductive step when $n \ge 2$) assume that the claims hold for $n -1$.
		
		Fix $T > 0$, and define the exterior $\sigma$-algebra and its corresponding conditional distribution as 
		\begin{align}
			\mathcal{F}_T = \sigma ( \mathcal{H}|_{(\{i\} \times [-T, T])^c}), \qquad \prob_{T} = \prob \big( \cdot | \mathcal{F}_{T} \big) 
		\end{align}
		Now let $a < b \in [-T, T]$.  Applying Theorem \ref{thm:int_parts} with the Hamiltonian $\Ham(x) = e^x$ and $F = 1$ gives
		\begin{align} 
			\nonumber
			\E \big[ \mathcal{H}_n(b) - \mathcal{H}_n(a)  \big] &= \E \big[ \E_T [ \mathcal{H}_n(b) - \mathcal{H}_n(a) -  (\mathcal{H}_n(T) + \mathcal{H}_n(-T))(b-a)/(2T) ]\big] \\
			&= \E   \int_{-T}^{T} \E_T \big[ e^{\mathcal{H}_{n + 1}(s) - \mathcal{H}_n(s)} - e^{\mathcal{H}_{n}(s) - \mathcal{H}_{n - 1}(s)} \big] \textrm{Cov}^{T, a, b} (s) \, ds   \\
			\label{E:Fubini}
			&= \int_{-T}^{T} \E \big[ e^{\mathcal{H}_{n + 1}(s) - \mathcal{H}_n(s)} - e^{\mathcal{H}_{n}(s) - \mathcal{H}_{n - 1}(s)} \big] \textrm{Cov}^{T, a, b} (s) \, ds \\
			\label{eqn:emoments_init}
			&= (\E e^{\mathcal{H}_{n + 1}(0) - \mathcal{H}_n(0)} - \E e^{\mathcal{H}_{n}(0) - \mathcal{H}_{n - 1}(0)}) \int_{-T}^{T} \textrm{Cov}^{T, a, b} (s)   \, ds.
		\end{align}
		In the first equality, we have used the symmetry $\E\mathcal H_n(T) = \E\mathcal H_n(-T)$. The existence of the unconditional expectations in this equality uses Lemma \ref{lemma:opt_tran}. The second equality follows from Theorem \ref{thm:int_parts}. For the third equality, we require that $\E e^{\mathcal{H}_{n}(s) - \mathcal{H}_{n - 1}(s)} < \infty$ to justify the use of Fubini's theorem. In the $n = 1$ case, this is equivalent to the claim that $\E e^{\mathcal{H}_{n}(s)} < \infty$ which follows from the tail bounds in Lemma \ref{lemma:kpzeqn_tails}. In the $n \ge 1$ case this uses the inductive hypothesis. Note that a priori, the second expectation $\E e^{\mathcal{H}_{n+1}(s) - \mathcal{H}_{n}(s)}$ could equal $+ \infty$. The equality \eqref{E:Fubini} rules this out, showing that the left-hand side of \eqref{E:exponential-gap} is well-defined.
		The final equality uses the stationarity of $\mathcal H$.
		
		Next, we have that $\E[\mathcal{H}_n(b) - \mathcal{H}_n(a)] = (a^2 - b^2)/(2t)$ by Lemma \ref{lemma:opt_tran}. Therefore after evaluating the covariance integral in \eqref{eqn:emoments_init}, we have:
		\begin{align*}
			- \frac{b^2}{2t} + \frac{a^2}{2t} &= \big(- \frac{b^2}{2} + \frac{a^2}{2} \big) \big( \E e^{\mathcal{H}_{i + 1}(0) - \mathcal{H}_i(0)} - \E^{\mathcal{H}_{i}(0) - \mathcal{H}_{i - 1}(0)} \big),
		\end{align*}
		which simplifies to \eqref{E:exponential-gap}.
	\end{proof}
	
	\section{$\Ham$-Brownian line ensembles with stationary Toda boundary} \label{sec:toda_bdary}

	In this section, we study an $\Ham$-Brownian line ensemble $\Lag = (\Lag_1 , \dots, \Lag_n)$ on the interval $[-T, T]$ with a parabolic lower boundary, and side boundary conditions chosen so that we can explicitly compute the most favourable configuration for $\Lag$. Our goal in the section will be to show concentration around this favourable configuration. Throughout this section we set $\Ham(x) = e^x$.
	Let
	\begin{align*}
		f_i(x) &= -\frac{x^2}{2t} + e_i, \qquad e_i := \sum_{j=1}^{i-1} \log\!\left(\frac{j}{t}\right) + \frac{t}{24},
	\end{align*}
	and let $\vec{e} = (e_1 , \ldots , e_n)$. We define
	\begin{align*}
		\Lag \sim \prob_{\Ham}^{t, n} := \prob_{\Ham}^{1, n, (-T, T), \vec{e} - T^2/(2t), \vec{e} - T^2/(2t), \infty, f_{n + 1}}.
	\end{align*}
	That is, $\Lag$ is an $\Ham$-Brownian line ensemble on $[-T, T]$ with no upper boundary, parabola lower boundary $f_{n+1}$, and $\Lag(-T) = \Lag(T) = \vec{e} - T^2/(2t)$. We suppress the superscript $n$ and $t$ on $\prob_{\Ham}^{t, n}$ when these values are clear from context. Now, let $\prob_{\textrm{free}} = \prob^{\vec{x}}_{\textrm{free}} =  \prob_{\textrm{free}}^{1, n, (-T, T), \vec{x} , \vec{x})}$ be the law of $n$ independent Brownian bridges $B$ on $[-T, T]$ with $B(-T) = B(T) = \vec{x}$. 
	The Radon-Nikodym derivative of $\Lag \sim \prob_\Ham$ against $ \prob^{\vec{e} - T^2/(2t)}_{\textrm{free}}$ is given by:
	\begin{align*}
		\frac{d \prob_{\Ham}}{d \prob^{\vec{e} - T^2/(2t)}_{\textrm{free}}}(\Lag)  = Z^{-1} \exp \Big( - \sum_{i = 1}^{n} \int_{-T}^{T} e^{(\Lag_{i + 1} - \Lag_{i})(s)} \,ds \Big), \qquad Z = \E_{\textrm{free}}^{\vec{e} - T^2/(2t)} \exp \left( - \sum_{i=1}^{n} \int_{-T}^{T} e^{B_{i+1} - B_i} \right),
	\end{align*} 
	where $\Lag_{n + 1} = B_{n+1} = f_{n + 1}$ is the parabolic lower boundary. Our goal in this section is to show that top line $\Lag_1$ cannot deviate too far above the deterministic function $f_1$. This concentration is much easier to study when we analyze the shifted configuration
	\begin{align} \label{eqn:shifted_config}
		(\hat \Lag_1, \dots, \hat \Lag_n, 0) = (\Lag_1, \dots, \Lag_{n+1}) - (f_1, \dots, f_{n+1}).
	\end{align}
	This configuration equals $0$ at $\pm T$, and has law which is absolutely continuous with respect to $\prob^{\vec{0}}_{\textrm{free}}$. The next lemma computes the Radon-Nikodym derivative.
	\begin{lemma}
		\label{L:translated-RND}	
		Let $\prob_{\hat{\Ham}_t} = \prob_{\hat{\Ham}_t, (n)}$ be the law of $\hat \Lag =  (\hat \Lag_1, \dots, \hat \Lag_n)$, and let $\Phi(x) = e^x - 1 - x$. Then
		$$
		\frac{d\prob_{{\hat{\Ham}_t}}}{d\prob^{\vec{0}}_{\textrm{free}}}(\hat \Lag) = \hat Z^{-1} \hat W_n(\hat \Lag), \qquad \hat W_n(g) := \exp \Big( - \frac{1}{t} \sum_{i = 1}^{n} \int_{-T}^{T} i \Phi (g_{i + 1} - g_i)   \Big),
		$$
		where $\hat Z = \E^{\vec{0}}_{\textrm{free}} \hat W(B)$ and $g_{n+1} = 0$. Moreover, with $Z$ as above we have that
		$$
		Z  = \hat Z \exp(-\mathcal{E}^{n, t}_{\min}), \qquad \mathcal{E}^{n, t}_{\min} :=\frac{nT^3}{3t^2}+ \frac{Tn(n+1)}{t}.
		$$
	\end{lemma}
	\begin{proof}
		We will first shift the endpoints of $\Lag$ and then shift by the parabola. Let
		$$
		(L_1, \dots, L_{n+1}) = (\Lag_1, \dots, \Lag_{n+1}) - \vec{e} + T^2/(2t),
		$$
		and let $\mu$ be the law of $L= (L_1, \dots, L_n)$.
		Then
		$$
		\frac{d\mu}{d\prob^{\vec{0}}_{\textrm{free}}}(L) = Z^{-1} \tilde W(L), \qquad \tilde W(h) = \exp \Big( - \sum_{i=1}^{n} \int_{-T}^T \frac{i}{t} \, e^{h_{i+1} - h_i} \Big), 
		$$
		where $h_{n+1}(x) = (T^2- x^2)/(2t) =: p(x)$, and $\E^{\vec{0}}_{\textrm{free}} \tilde W(B)$ is still equal to the original partition function $Z$. We now shift by the parabola $p$. First, let $\prob_B$ be the law of a Brownian bridge $B$ on $[-T, T]$ with $B(\pm T) = 0$ and let $\prob_{B-p}$ be the law of $\hat B = B - p$. By the Cameron-Martin theorem we have
		\begin{equation}
			\label{E:cam-mart}
			\frac{d\prob_{B - p}}{d\prob_{B}}(\hat B)= \exp\left( -\int_{-T}^{T} \dot{p}_i \, d \hat B_i - \frac12 \sum_{i=1}^n \int_{-T}^{T} \dot{p}_i^2 \right) = \exp\left(-\frac{T^3}{3 t^2} - \frac{1}{t} \int_{-T}^{T} \hat B_i(x) dx \right),
		\end{equation}
		where in the equality we have used stochastic integration by parts to rewrite the stochastic integral as a Lebesgue integral. Next, letting $\mu^k$ denote the $k$-fold product of a measure $\mu$, we have that
		\begin{equation*}
			\frac{d\prob_{{\hat{\Ham}_t}}}{d\prob_B^k}(\hat \Lag) = 	\frac{d\prob_{{\hat{\Ham}_t}}}{d\prob_{B-p}^k}(\hat \Lag) 	\frac{d\prob_{B-p}^k}{d\prob_{B}^k}(\hat \Lag) = \frac{d\mu}{d\prob_B^k}(\hat \Lag + p) 	\frac{d\prob_{B-p}^k}{d\prob_{B}^k}(\hat \Lag),
		\end{equation*}
		and so
		$$
		\frac{d\prob_{{\hat{\Ham}_t}}}{d\prob_B^k}(\hat \Lag) = Z^{-1} \exp \Big( - \sum_{i=1}^{n} \int_{-T}^T \frac{i}{t} \, e^{\hat \Lag_{i+1} - \hat \Lag_i} \Big) \exp\left(-\frac{T^3}{3 t^2} - \frac{1}{t} \int_{-T}^{T} \hat \Lag_i(x) dx \right),
		$$
		where $\hat \Lag_{n+1} = 0$.
		Now, we can write
		$$
		\sum_{i=1}^{n} i e^{\hat \Lag_{i+1} - \hat \Lag_i} +  \hat \Lag_i = \frac{n(n+1)}{2} + \sum_{i=1}^n i \Phi(\hat \Lag_{i+1} - \hat \Lag_i),
		$$
		which allows us to simplify the previous display to get
		\begin{equation}
			\label{E:Z-final}
			\frac{d\prob_{{\hat{\Ham}_t}}}{d\prob_B^k}(\hat \Lag) = Z^{-1} \exp (-\mathcal{E}^{n, t}_{\min}) \exp \Big( - \frac{1}{t} \sum_{i = 1}^{n} \int_{-T}^{T} i \Phi (\hat \Lag_{i + 1} - \hat \Lag_i)(x) \, dx   \Big).
		\end{equation}
		Since both sides above integrate to $1$ under the measure $\prob_B^k = \prob^{\vec{0}}_{\textrm{free}}$, we get the desired evaluation for the partition function $Z$. Setting $\hat Z = Z \exp(\mathcal{E}^{n, t}_{\min})$ gives the desired Radon-Nikodym derivative formula. 
	\end{proof}
	
	The remainder of the section is devoted to controlling the behaviour of the top line of the ensemble $\hat \Lag$. We begin with a simple lower bound on $\hat{Z}$. 
	
	\begin{lemma} \label{lemma:part_lowbnd} There exists a  universal constant $d > 0$ so that
		$$
		\hat{Z}  \geq \frac{1}{2} \exp \left( -  d T n \max \left( 1 ,  t^{-1/2} n^{1/2} \right) \right).
		$$
	\end{lemma}
	
	\begin{proof}
		Fix $r > 0$, and let
		\begin{align*}
			\mathcal S_r &= \Big\{  \sup_{i = 1, \ldots, n} \sup_{s \in [-T, T]} |B_i(s)| \leq r \Big\}
		\end{align*}
		be the event where $n$ Brownian bridges each stay in a tube of width $r$. By Lemma \ref{lemma_tube}, there exists an absolute constants $d > 0$ such that $\prob^{\vec{0}}_{\textrm{free}}(\mathcal{S}_r) \geq \frac{1}{2} \exp (- \frac{d T n}{r^2})$ for all $r > 0 $. In addition, on $\mathcal S_r$, we have $|B_{i + 1} - B_i| \leq 2r$, and therefore $\Phi(B_{i + 1} - B_i) \leq e^{2r} - 1 - 2r \leq r^2$ for $r \leq 1$. Therefore for $r \leq 1$, $\hat{Z}$ is lower bounded by 
		\begin{align*}
			\E^{\vec{0}}_{\textrm{free}} \Big[ \mathbbm{1}_{\mathcal{S}_r} \exp \Big( - t^{-1} \sum_{i = 1}^{n} \int_{-T}^{T} i \Phi (B_{i + 1} - B_i) \Big) \Big]
			&\geq \prob^{\vec{0}}_{\textrm{free}} (\mathcal{S}_r) \cdot \exp \Big( - t^{-1} \sum_{i = 1}^{n} \int_{-T}^{T} i r^2  \Big) \\ 
			&\geq  \frac{1}{2} \exp \Big( - \frac{dT n }{r^2} -  \frac{T n (n + 1) r^2}{t}  \Big)
		\end{align*}
		Now set $r = \max \left( 1 , t^{1/4} n^{-1/4} \right)$. This gives the desired bound after a short computation.
	\end{proof} 
	
	Since $\hat W \le 1$, Lemma \ref{lemma:part_lowbnd} has the following immediate corollary.
	
	\begin{cor} \label{cor:trivial_case}
		There exists a universal constant $d > 0$ such that for any measurable event $A$, 
		\begin{align*}
			\prob_{\hat{\Ham}_t} \left( A \right) & \leq \hat{Z}^{-1} \prob^{\vec{0}}_{\textrm{free}}(A) \le 2\exp \left(d T n \max \left( 1 ,  t^{-1/2} n^{1/2} \right) \right) \prob^{\vec{0}}_{\textrm{free}}(A).
		\end{align*}
	\end{cor}
	
	
	We are now ready to bound the supremum of the top line $\hat L_1$.
	\begin{lemma} \label{lemma:top_line_melon}
		For all $\epsilon > 0$, there exist constants $c, d > 0$ so that for any $t, T > 0$, $n \in \N$ and $a > c T n^{\epsilon} \max (1, t^{-1/4} n^{1/4} )$ we have
		\begin{align*}
			\prob_{\hat{\Ham}_t} (\sup_{x \in [-T, T]}  \hat{\Lag}_1(x) \geq a) & \leq \exp (- d a^2/T).
		\end{align*}
	\end{lemma}
	
	\begin{proof}
		Setting some notation, we define
		\begin{align*}
			p_n(a) = \prob_{\hat{\Ham}_t, (n)} (\sup_{x \in [-T, T]}  \hat{\Lag}_1(x) \geq a),
		\end{align*}
		and set $M_n = \max (1, t^{-1/4} n^{1/4} )$. We start by proving the following claim.\\
		
		\noindent \textbf{Claim:} \qquad There exists an absolute constant $c > 0$ such that for any choices of $T, t > 0 $ and $n \in \N, \ell \in \{2, \dots, \lfloor n/2 \rfloor\}$ , for any $a > c T M_n \sqrt{n/\ell}$, we have that
		\begin{equation}
			\label{E:pnpell}
			p_n(a) \le \max_{\ell \le k \le 2 \ell} p_{k-1}(a/2) + \exp(-a^2/(2T)).
		\end{equation}
		
		\noindent \textbf{Proof of claim:} \qquad For any $\ell \in \{2, \ldots, \lfloor n/2 \rfloor \}$ consider the event
		\begin{align*}
			A_{\ell} = \bigcup_{k=\ell}^{2\ell} D_k, \qquad D_k \defeq \left\{ \sup_{x \in [-T, T]} \hat{\Lag}_k(x) \leq \frac{a}{2} \right\}. 
		\end{align*}
		By Corollary \ref{cor:trivial_case}, we have
		\begin{align*}
			\prob_{\hat{\Ham}_t} (A_{\ell}^c) & \leq \hat{Z}^{-1} \prob^{\vec{0}}_{\textrm{free}} (A_{\ell}^c) \le \exp(d T n M_n^2 - \frac{a^2 \ell}{T}) \le \exp(-a^2/(2T)),
		\end{align*}
		where the final inequality uses the lower bound on $a$ and the upper bound on $n/\ell$, choosing $c > 0$ sufficiently large.
		Therefore
		\begin{align} \label{easy_step}
			p_n(a) & \leq \prob_{\hat{\Ham}_t} \left( \Big\{ \sup_{x \in [-T, T]}  \hat{\Lag}_1(x) \geq a  \Big\} \cap  A_{\ell} \right) +  \exp(-a^2/(2T)).
		\end{align}
		Now, \eqref{easy_step} is upper bounded by
		\begin{align} \label{union_bound}
			\max_{\ell \le k \le 2 \ell} \prob_{\hat{\Ham}_t} \Big(\sup_{x \in [-T, T]}  \hat{\Lag}_1(x) \geq a  \, \Big| \, D_k \Big) +  \exp(-a^2/(2T)).
		\end{align}		
		Now fix $k \in \{ \ell , \ldots , 2 \ell\}$, and consider a line ensemble $\tilde \Lag = (\tilde \Lag_1, \dots, \tilde \Lag_{k-1})$ such that $\tilde \Lag - a/2 \sim \prob_{\hat \Ham_t, (k-1)}$. That is, $\tilde{\Lag}$ has the same distribution as $\prob_{\hat \Ham_t, (k-1)}$, translated upwards by $\frac{a}{2}$. By Lemmas \ref{lemma:stoch_dom1}, \ref{lemma:stoch_dom2}, if $\hat{\Lag} \sim \prob_{\hat \Ham_t}(\cdot \;|\; D_k)$, then $\tilde{\Lag}$ stochastically dominates $(\hat{\Lag}_1, \dots, \hat{\Lag}_{k-1})$. Therefore \eqref{union_bound} is upper bounded by the right-hand side of \eqref{E:pnpell}, as desired. \hfill $\blacksquare$ \\

		Given the claim, the lemma follows by discrete calculus. Indeed, for fixed $\epsilon > 0$, construct a sequence of integers $\ell_1 = 2 < \dots < \ell_{\lceil 2/\epsilon \rceil} = n$ such that $2\ell_i/\ell_{i-1} \le n^\epsilon$ for all $n$. We can do this as long as $n$ is sufficiently large given $\epsilon$. Then for $a > c 2^{\lceil 2/\epsilon \rceil} T M_n n^\epsilon$, for any $i \in \{\ell_j -1, \dots 2 \ell_j - 1\}$, since $i/\ell_{j - 1} \leq 2 \ell_{j} / \ell_{j - 1} \leq n^{\epsilon}$, taking $(n, \ell, a)  = (i, \ell_{j  - 1}, a/2^{\lceil 2/\epsilon \rceil - j})$ in the above claim gives that 
		\begin{equation*}
			\label{E:pnpell}
			p_i(a/2^{\lceil 2/\epsilon \rceil - j}) \le \max_{\ell_{j-1} \le k \le 2 \ell_{j-1}} p_{k-1}(a/2^{{\lceil 2/\epsilon \rceil} - j + 1}) + \exp(-a^2/(2^{2{\lceil 2/\epsilon \rceil} - 2j + 1}T)).
		\end{equation*}
		Chaining together these inequalities with initial tail bounds on $p_1, p_2, p_3$ from Corollary \ref{cor:trivial_case} gives the desired upper bound on $p_n(a) = p_{\ell_{\lceil 1/\epsilon\rceil}}(a)$.
	\end{proof}

	\section{Bounding the KPZ line ensemble}
	\label{sec:bounding-final}
	In this section, we prove Theorem \ref{thm:main_thm}. Recall from the statement of that theorem that
	\begin{align*}
		e_n^{(t)} &= \frac{t}{24} + \sum_{i = 1}^{n-1} \log( \frac{i}{t} )  = \log (t^{- (n - 1)} (n - 1)! ) + \frac{t}{24} = n \log n - n - (n - 1) \log t + \frac{t}{24} + O(\log n).
	\end{align*}
	Our goal is to show that $\mathcal{H}^{(t)}_n = \mathcal{H}_n$, the $n$th line of the KPZ$^{(t)}$ line ensemble, concentrates around the shifted parabola $e_n^{(t)} - x^2/(2t)$. With this in mind, define
	\begin{align*}
		\hat{\mathcal{H}}_n(x) = \mathcal{H}_n(x) - e_n^{(t)} + \frac{x^2}{2t},
	\end{align*}
	and for $x, T> 0$ and $T > 0$, define
	\begin{align*}
		p_{n,t}^{-}(x) &= \prob \left(\hat{\mathcal{H}}_n(0) \leq - x \right) ,\\
		p_{n,t}^{+}(x) &= \prob \left(\hat{\mathcal{H}}_n(0) \geq  x \right) ,\\
		q_{n,t}^{-}(x, T) &= \prob \left( \sup_{y \in [-T, T]} \hat{\mathcal{H}}_n(y) \leq - x \right) ,\\
		q_{n,t}^{+}(x, T) &= \prob \left( \inf_{y \in [-T, T]} \hat{\mathcal{H}}_n(y) \geq  x \right).
	\end{align*}
	We will prove Theorem \ref{thm:main_thm} by deriving inequalities between the four functions above. Before doing so, we record an a priori bound on $p_{n, t}^+$ which follows from the exponential moment formula for the gaps in the KPZ line ensemble, Theorem \ref{thm:e_moments}.

	\begin{lemma} \label{lemma:endpts_expmm} Let $n \in \N$ and $t > t_0 > 0$. Then for all $x > 0$,
		\begin{align*}
			p_{n,t}^{+}(x) & \leq p_{1, t}^+(x/2) + (n-1) e^{-x/(2n)} \le 2\exp(-d x^{3/2} t^{-1/2}) + n \exp(-x/(2n)),
		\end{align*}
		where $d > 0$ depends only on $t_0$.
	\end{lemma}
	
	\begin{proof} Notice that  
		\begin{align*} 
			\left\{ \hat{\mathcal{H}}_n(0) \geq x + e_{n}^{(t)} \right\} \subset 	\left\{ \hat{\mathcal{H}}_1(0) \geq  \frac{x}{2}  \right\} \cup \bigcup_{i = 1}^{n - 1}   \left\{  \mathcal{H}_{i + 1}(x) - \mathcal{H}_i(x) > \log(\frac{i}{t}) +  \frac{x}{2n}   \right\} .
		\end{align*}
		Thus, by Theorem \ref{thm:e_moments}, a union bound, and Markov's inequality we have 
		\begin{align*}
			p_{n,t}^{+}(x)  & \leq \prob \big(\hat{\mathcal{H}}_1(0)\geq  \frac{x}{2}  \big) + \sum_{i = 1}^{n-1} \prob \big(  \mathcal{H}_{i + 1}(x) - \mathcal{H}_i(x) > \log(\frac{i}{t}) +  \frac{x}{2n}   \big)  \\
			&\leq p_{1, t}^+(x/2) + (n-1) e^{-x/(2n)}.
		\end{align*}
		The second inequality uses Lemma \ref{lemma:kpzeqn_tails}.
	\end{proof}
	
	Next, we obtain inequalities for the functions $q^\pm, p^\pm$ as follows. First, we can upper bound $q_{n,t}^{-}$ using $p_{n,t}^{+}$. Next, we bound $p_{n,t}^{-}$ from $q_{n,t}^{-}$. By a symmetric argument, we are able to bound $q_{n,t}^{+}$ from $p_{n,t}^{-}$ and then bound $p_{n,t}^{+}$ from $q_{n,t}^{-}$. Closing this loop gives self-referential inequalities for $p_{n, t}^\pm$ which together prove Theorem \ref{thm:main_thm}. The next lemma proves our building block inequalities. 
	
	\begin{lemma} \label{lemma:curv_adj_lemma}
		First, for any $a, T, t > 0$ and $n \in \N$ we have 
		\begin{align}
			p_{n,t}^{\pm}(a) \leq 2\exp(- d a^2/T) + q_{n,t}^{\pm} (\frac{a}{2}, T), \label{eq:second_ineq}
		\end{align}
		where $d > 0$ is a universal constant.
		Next, for any $\epsilon > 0$ there exist constants $c, d > 0$ depending only on $\epsilon$ such that the following holds. For any $T, a, \kappa, t > 0$ and $n \in \N$ satisfying
		\begin{align}
			\label{E:param-ranges}
			4 t^{1/2} (a + \kappa)^{1/2} (1 + n/a)^{1/2} &\leq T \leq \frac{a}{c n^{\epsilon} M }
		\end{align} 
		where $M = \max(1, t^{-1/4} n^{1/4})$, we have:
		\begin{align}
			q_{n,t}^{\pm} (a, T) & \leq \exp \left(  - d a^2 / T \right) + p_{1, t}^{\pm}(\frac{a}{4}) + 2 \sum_{i = 1}^{n - 1} p_{i , t}^{\mp}(\kappa). \label{eq:first_ineq}
		\end{align}
	\end{lemma}
	
	\begin{proof}
		To prove \eqref{eq:second_ineq}, notice that
		\begin{align}
			\left\{   \inf_{x \in [-T, T]} \pm \hat{\mathcal{H}}_n(x)  \leq -a  \right\}
			&\subset  \left\{  \sup_{u, v \in [-T, T] , \, u \neq v} |\hat{\mathcal{H}}_n(u) - \hat{\mathcal{H}}_n(v)| \leq \frac{a}{2}   \right\} \cup  \left\{   \sup_{x \in [-T, T]} \pm \hat{\mathcal{H}}_n(x) \leq- \frac{a}{2} \right\}.
		\end{align}
		The inequality \eqref{eq:second_ineq} then follows from applying Lemma \ref{lemma:opt_tran} to the first event on the right-hand side above.
		
		Now we proceed to the proof of \eqref{eq:first_ineq}. We prove the $-$ case only as the $+$ case follows symmetrically. Define the events
		\begin{align}
			A_{T, a}^{n } \defeq   \{ \sup_{x \in [-T, T]}  \hat{\mathcal{H}}_{n}(x) \leq - a  \} , \qquad  \mathfrak{E}^{n}_{T, \kappa} \defeq \bigcap_{i = 1}^{n - 1}  \left\{ \hat{\mathcal{H}}_i(T) \leq \kappa  \textrm{ and }  \hat{\mathcal{H}}_i(-T) \leq   \kappa \right\}  .
		\end{align}
		By a union bound,
		\begin{align} \label{eq:first_step}
			q_{n,t}^{-}(a,T) & \leq \prob \left( A_{T,a}^n  \cap \mathfrak{E}^{n}_{T, \kappa} \cap \left\{ \hat{\mathcal{H}}_1(0) \geq - \frac{a}{4} \right\}  \right)    +  2 \sum_{i = 1}^{n - 1} p_{n , t}^{+}(\kappa) + p_{1, t}^{-}(\frac{a}{4}).
		\end{align}
		where the summation in the above inequality uses the stationarity of process $\mathcal{H}_i(x) + \frac{x^2}{2t}$. Thus, it only remains to bound the first term of the right-hand side of the above inequality, which we do by an $\Ham$-Brownian Gibbs argument. First, define the exterior $\sigma$-algebra and its corresponding conditional distribution as
		\begin{align}
			\mathcal{F}_{\textrm{ext}} & \defeq \sigma \Big\{ \mathcal{H}_i(x) \, : \, (i, x) \not\in \{1, \dots, n-1\} \times [-T, T] \Big\} , \qquad \prob_{\Ham, \mathcal{F}_{\textrm{ext}}} (A) \defeq \E_{\Ham}  [\mathbbm{1}_{A} \, | \, \mathcal{F}_{\textrm{ext}}].
		\end{align} 	
		Since $A^{n}_{T, a}$ and $\mathfrak{E}^n_{T, \kappa}$ are $\mathcal{F}_{\textrm{ext}}$ measurable, it follows that  
		\begin{align} \label{eqn:first_raise}
			\prob \left( A_{T,a}^n  \cap \mathfrak{E}^{n}_{T, \kappa} \cap \left\{ \hat{\mathcal{H}}_1(0) \geq - \frac{a}{4} \right\}  \right)   &= \E_{\Ham} \Big[  \prob_{\Ham, \mathcal{F}_{\ext}} (\hat{\mathcal{H}}_1(0) \geq -\frac{a}{4}) \mathbbm{1} (A_{T, a}^{n} \cap \mathfrak{E}^{n}_{T, \kappa} )    \Big].
		\end{align}
		Now define 
		\begin{align}
			\nonumber
			f_{n}(x) &=  -\frac{x^2}{2t} + e_{n}^{(t)} - a ,  \\
			\nonumber
			v_i &=  -\frac{T^2}{2t} + e_i^{(t)} + \kappa , \qquad  \vec{v} = (v_1 , \ldots , v_{n-1}) .
		\end{align} 
		Notice that on the event $A^{n }_{T, a} \cap \mathfrak{E}^{n, \kappa}$, it holds that $\mathcal{H}_{i}(\pm T) \leq v_i$ and $\mathcal{H}_{n} \leq f_{n}$. Thus, by Lemmas \ref{lemma:stoch_dom1} and \ref{lemma:stoch_dom2}, the term inside the expectation in \eqref{eqn:first_raise} is bounded above by
		\begin{align}\label{eqn:Hprob}
			\prob_{\Ham}^{1, n-1, (-T, T), \vec{v}, \vec{v}, \infty, f_{n}} \big( \mathcal L_1(0) \geq \frac{t}{24} -\frac{a}{4} \big).
		\end{align}
		We now apply the following procedure to raise the boundary data. First, we decrease the curvature of the boundary parabola $f_{n}$ to a new parabola $g_{n}$ so that the gap $g_{n }(\pm T) - v_{n-1}$ exactly matches the Toda spacing from Section \ref{sec:toda_bdary}. More precisely, define $t' > t$ by the implicit equation
		\begin{align}
			- \frac{T^2}{2t'} + e_{n}^{(t)}  - a-    \log(\frac{n - 1}{t'}) =  v_{n-1},
		\end{align}
		or equivalently
		\begin{align} \label{eqn:curv_lwer}
			\frac{t' - t}{2 t t'} &= \frac{\kappa + a - \log (\frac{t'}{t})}{T^2}.
		\end{align}
		The lower bound on $T$ guarantees $T^2 > 16t (\kappa + a)$ which implies that this equation has a unique solution $t' \in [t, 2t]$. To see this, let $u = t'/t$. Then \eqref{eqn:curv_lwer} is equivalent to
		\begin{align*}
			\frac{1}{u} - \frac{2 t}{T^2} \log u = 1 - \frac{2t (\kappa + a)}{T^2} \in (\frac{1}{2} , 1). 
		\end{align*}
		Existence and uniqueness of $u$ satisfying the above equation then follows since the function $h(u) = \frac{1}{u} - \frac{2t}{T^2} \log u$ is strictly decreasing for $u \in [1, 2]$, and $h(1) = 1, h(2) < 1/2$.
		
		Second, we raise the boundary points $v_1, \ldots , v_{n - 1}$ by decreasing the gaps $v_{i+1} - v_i$, while keeping $v_{n-1}$ fixed, again to match the gap spacing of Section \ref{sec:toda_bdary}. This procedure yields the following raised boundary conditions:
		\begin{align}  
			v_{n-1}' &= v_{n-1}, \qquad v_i' = v_{i+1}' -  \log (\frac{i}{t'}), \qquad \vec{v}' = (v_1' , \ldots , v_{n-1}'),\\
			g_{n}(x) &=  -  \frac{x^2 }{2t'} + e_{n}^{(t)} - a.
		\end{align}
		By Lemmas \ref{lemma:stoch_dom1} and \ref{lemma:stoch_dom2}, the $\Ham$-Brownian line ensemble $\mathcal L'$ with the raised boundary data $(\vec{v}' , \vec{v}', g)$ stochastically dominates the $\Ham$-Brownian line ensemble $\Lag$ with boundary data $(\vec{v} , \vec{v}, f)$, and therefore \eqref{eqn:Hprob} is bounded above by
		\begin{align} \label{eqn:app_of_prev_sec}
			\prob_{\Ham}^{1, n,(-T, T), \vec{v}', \vec{v}',  \infty. g_n} \big( \Lag_1'(0) \geq \frac{t}{24} -\frac{a}{4} \big) .
		\end{align}
		We now wish to apply Lemma \ref{lemma:top_line_melon} to show the above probability is rare. First, note that the parabola where we expect to the first curve to sit is now given by
		\begin{align*}
		g_1(x) :=g_n(x) - \sum_{i=1}^{n-1} \log(\frac{i}{t'} )
		 &= \frac{t}{24} - \frac{x^2}{2t} - a + (n-1) \log (\frac{t'}{t})
		\end{align*}
		Now, we have that
		\begin{align*}
			(n-1) \log (t'/t) \leq \frac{n(t' - t)}{t} = \frac{2 t'n(\kappa + a - \log (t'/t))}{T^2} \le \frac{a}{4}. 
		\end{align*}
		Here the equality uses the definition \eqref{eqn:curv_lwer} and the last inequality uses that $t' \in [t, 2t]$ and $T^2 > 16 t n (\kappa + a)/a$.
		Hence, we have that \eqref{eqn:app_of_prev_sec} is bounded above by
		\begin{align*}
			\prob_{\Ham}^{1, n,(-T, T), \vec{v}', \vec{v}',  \infty, g_n} \big( (\Lag_1'  - g_1)(0) \geq \frac{a}{2} \big).
		\end{align*}
		Since $a \geq c T n^{\epsilon} M$, we may now apply Lemma \ref{lemma:top_line_melon} to conclude the proof.
	\end{proof}
	
	Using this lemma, we may now prove Theorem \ref{thm:main_thm}. We restate it here with the present notation.

	\begin{thm} \label{theorem:tails_kpzLE_sharp}
		Fix $\epsilon > 0, t_0 > 0$.  There exists constants $d, C > 0$ so that for all $n \in \N, t > t_0$ and $a \geq a_0 := C (t n^{2\epsilon} + t^{1/4} n^{3/4 + 2\epsilon})$ we have
		\begin{align*}
			\prob \left(  | \hat{\mathcal{H}}_n(0) | \geq a \right) & \leq 2 \exp(- d a M n^{\epsilon}).
		\end{align*}
		Here $M = \max(1, t^{-1/4} n^{1/4})$.
	\end{thm}
	
	In Theorem \ref{theorem:tails_kpzLE_sharp}, the $n^\epsilon$ lower bounds are suboptimal, and could be improved by sharpening the proof. We have not attempted this, as we do not expect that this gives a sharp result in any case.
	
	\begin{proof}
		Let $n \in \N, t > t_0 > 0, \epsilon > 0$. In what follows, constants $c, d > 0$ are constants that may change from line to line and are allowed to depend on $\epsilon$ and $t_0$, but on no other parameters. We first prove a bound on the upper tail, i.e.\ a bound on $p_{n, t}^+(a)$. We say that a triple $(T, a, \kappa)$ is $(n, t)$-\emph{good} if the inequalities \eqref{E:param-ranges} hold for the given $n, t, \epsilon$. Now, by using the four inequalities in Lemma \ref{lemma:curv_adj_lemma}, we have that
		\begin{align*}
			p_{n,t}^{+}(a) & \leq 2 \exp(- d a^2/T) + q_{n,t}^{+}(\frac{a}{2}, T) \\
			& \leq p_{1,t}^{+}(\frac{a}{8}) + 2 \exp(-d a^2/T) 	+ 2 \sum_{i=1}^{n - 1} p_{i,t}^{-} (\kappa)		\\
			& \leq p_{1,t}^{+}(\frac{a}{8}) + 2 \exp(- d a^2/T) + 2 n \exp(- d \kappa^2/L) + \sum_{i = 1}^{n - 1} q_{i, t}^{-} (\frac{\kappa}{2} , L)  \\   
			& \leq p_{1,t}^{+} (\frac{a}{8}) + 2 n p_{1, t}^{-} (\frac{\kappa}{8}) + 2 \exp(- d a^2 / T) + 2 n \exp(- d \kappa^2 / L) + 2 \sum_{i = 1}^{n - 1} \sum_{j = 1}^{i - 1} p_{j, t}^{+}(\sigma).
		\end{align*}
		Here the first inequality uses \eqref{eq:second_ineq} for $p_{n, t}^+$, the second inequality uses \eqref{eq:first_ineq} for $q_{n, t}^+$ and requires that the triple $(T, a/2, \kappa)$ be $(n, t)$-good. The third inequality uses \eqref{eq:second_ineq} for $p_{i, t}^-, i = 1, \dots, n-1$ and the fourth inequality uses \eqref{eq:first_ineq} for $q_{i, t}^+, i = 1, \dots, n-1$ and requires that the triple $(L, \kappa/2, \sigma)$ be $(i, t)$-good for all $i = 1, \dots, n-1$. Noting that any $(n, t)$-good triple is $(i, t)$-good for all $i \le n$, the above chain of inequalities holds as long as both of the triples $(T, a/2, \kappa), (L, \kappa/2, \sigma)$ are $(n, t)$-good. 
		
		Before setting $T, L, \kappa, \sigma$ observe that we can make one more simplification to the computation. If we set $P_{n, t}^+(a) = \max_{i =1 , \dots, n} p_{i, t}^+(a)$, then observe that subject to the same constraints on $a, T, L, \kappa, \sigma$ we have that
		\begin{equation}
			\label{E:somewhat-easier}
			P_{n, t}^+(a) \le p_{1,t}^{+} (\frac{a}{8}) + 2 n p_{1, t}^{-} (\frac{\kappa}{8}) + 2 \exp(- d a^2 / T) + 2 n  \exp(- d \kappa^2 / L) + n^2 P_{n, t}^{+}(\sigma).
		\end{equation}
		To see why \eqref{E:somewhat-easier} holds, we use that that $	P_{n, t}^+(a) = p_{j, t}^+(a)$ for some $j \le n$, and then apply the above inequality chain to $p_{j, t}^+$, replacing the final sum of $p_{j, t}^+$-terms with the larger $P_{n, t}^+$-terms. The constraints on $a, T, L, \kappa, \sigma$ still apply since the property of being $(i, t)$-good becomes stronger with increasing $i$, as mentioned previously. 
		
		Examining \eqref{E:somewhat-easier}, one can guess that under the right parameter choices, it transfers an inequality from the deep tail $P_{n, t}^+(\sigma)$ over to a shallower tail $P_{n, t}^+(a)$ for $a < \sigma$. We also see that the inequality becomes stronger the more we increase $\kappa, \sigma$ so we will set these maximally given the second constraint in \eqref{E:param-ranges}. For a large constant $c > 0$, let
		\begin{align*}
			T &= \frac{a}{c n^{\epsilon} M} , \qquad \kappa = \frac{1}{10^3 c^2} \frac{a^2  }{t n^{2 \epsilon} M^2 (1 + n/a)} , \\
			L &= \frac{1}{c} \frac{\kappa}{n^{\epsilon} M}  , \qquad \sigma = \frac{1}{10^3 c^2} \frac{\kappa^2 }{t n^{ 2 \epsilon} M^2 (1 + n/a)} .
		\end{align*} 
		We have that $\kappa \ge 2a$ and $\sigma \ge 2 \kappa$ under the assumption that 
		$$
		a > C (t n^{2\epsilon} M^2 + \sqrt{t} n^{1/2 + \epsilon} M),
		$$
		for $C > 0$ sufficiently large given $c, \epsilon, t_0$. The right-hand side above is bounded above by $a_0$ in the statement of the theorem. Moreover, given that $\kappa \ge 2a$, $\sigma \ge 2 \kappa$, it is straightforward to check that the triples $(T, a/2, \kappa), (L, \kappa/2, \sigma)$ are $(n, t)$-good. Finally, $a^2/T \leq \kappa^2/L$ since $\kappa \ge 2 a$. Hence, equation \eqref{E:somewhat-easier} simplifies to:
		\begin{equation*}
			P_{n,t}^{+}(a) \leq (2n + 1) p_{1,t}^{+}(\frac{a}{8}) + (2n + 1) \exp(-d a^2/T) + n^2 P_{n, t}^{+}(\sigma). 
		\end{equation*}
		Now, the first two terms on the right-hand side above are both bounded above by $2 \exp(-d a n^{\epsilon} M)$. For the first term, this uses Lemma \ref{lemma:kpzeqn_tails} and the lower bound on $a$. This gives
		\begin{equation}
			P_{n,t}^{+}(a) \leq 2 \exp(-d a n^{\epsilon} M ) + n^2 P_{n, t}^{+}(\sigma) \leq 2 \exp(-d a n^{\epsilon} M ) + n^2 P_{n, t}^{+}(4a).
			\label{E:before_iteration}
		\end{equation}

		Next, let $\ell = \lceil 3 \log_4 n \rceil$. By iterating the above inequality $\ell$ times, we obtain
		\begin{align*}
			P_{n,t}^{+}(a) \leq c n^{2\ell}  \exp(-d a n^{\epsilon} M ) + c n^{2\ell} P_{n, t}^{+}(a n^3).
		\end{align*}
		The second term on the right-hand side above can be bounded by Lemma \ref{lemma:endpts_expmm}, and is easily seen to be lower order when compared with the first term. The bound on $p_{n, t}^+(a)$ in the lemma then follows by simplifying through decreasing $d$, noting that the prefactor $n^{2 \ell}$ is lower order when compared to the exponent.
		
		We now bound $p_{n, t}^-(a)$ through the bounds on $p_{n, t}^+$. Indeed, similar to the computation at the beginning of the lemma, we have that for any triple $(T, a/2, \kappa)$ which is $(n, t)$-good, we have
		\begin{align*}
		p_{n,t}^{-}(a) & \leq 2 \exp(- d a^2/T) + q_{n,t}^{-}(\frac{a}{2}, T) \\
& \leq p_{1,t}^{-}(\frac{a}{8}) + 2 \exp(-d a^2/T) 	+ 2 \sum_{i=1}^{n - 1} p_{i,t}^{+} (\kappa).
		\end{align*}
Choose $T, \kappa$ and let $a > a_0$ so that $\kappa > 2 a$. Then as in our previous computations, we have that
$$
p_{n,t}^{-}(a) \le 2 \exp(-d a n^{\epsilon} M ) + n P_{n,t}^{+} (2a).
$$ 		
Applying the previously established bounds on $P_{i, t}^+$ then yields the result.
	\end{proof}
	
	\begin{proof}[Proof of Corollary \ref{C:mod-plus-moment}] We use the $\hat{\mathcal{H}}_n, M$ notation from previous proofs. Let $I = \mathbb Z \cap [-e^{a}, e^a]$. Then
	\begin{align*}
			\prob \Big(  &\sup_{x \in [-e^a, e^{a}]} | \hat{\mathcal{H}}_n(x) | > a \Big) \\			
			&\le \sum_{i \in I} \prob \left(| \hat{\mathcal{H}}_n(i) | > a/2 \right) + \prob \left( \exists s < t < s + 1 \in [-e^a, e^a] : |\hat{\mathcal{H}}_n(s) - \hat{\mathcal{H}}_n(t)| > a/2\right) \\
			&\le 3e^{a-d a n^\epsilon M} +e^{-d a^2}. 
	\end{align*}
The second inequality uses Theorem \ref{thm:main_thm} and stationarity of $\hat{\mathcal{H}}_n$ to bound the sum over $i \in I$, and Lemma \ref{lemma:opt_tran} for the second term. Simplifying the final result using the lower bound on $a$ gives the corollary.
	\end{proof}

	\appendix
	
	\section{Proof of monotone couplings for vector-valued Hamiltonians}
	
	\begin{lemma}  Consider two sets of boundary functions $f^{(1)}, f^{(2)}: (a, b) \to \R$ such that $f^{(1)} \leq f^{(2)}$, as well as endpoint pairs $ (x^{(i)}_j, y^{(i)}_j )_{j = 1}^k$ for $i \in \{1, 2\}$ such that $x^{(1)}_j \leq x^{(2)}_{j} $ and $y^{(1)}_{j} \leq y^{(2)}_j$ for all $j \in \{ 1, \ldots, k\}$. In addition, fix a vector of $n$ Hamiltonians $\mathbb{H} = (\mathcal{H}_1 , \ldots, \mathcal{H}_k)$ such that each $\mathcal{H}_i: \R \to [0, \infty)$ is convex. Let $\Lag^{(i)} = (\Lag_1^{(i)} , \ldots, \Lag_k^{(i)})$ be the collection of lines distributed according to $\prob^{(i)}_{\mathbb{H}}$. There exists a coupling of $\Lag^{(1)}, \Lag^{(2)}$ such that $\Lag^{(1)}_j \leq \Lag^{(2)}_j$ for all $j$.
	\end{lemma}
	
	\begin{proof} The proof follows almost exactly from the proof of Lemma 2.6 of \cite{corwin2016kpz} (or Lemma 5.7 of \cite{dimitrov2021characterization}) so we only provide a summary of the proof here. Let $\vec{x}^n = (x^n_1 , \ldots  , x^n_k)$ and $\vec{y}^n = (y^n_1 , \ldots  , y^n_k)$ such that $x^n_j, y^n_j \in n^{-1} (2 \Z)$ and $x^n_j \to x_j$, $y^n_j \to y_j$. We now define a set of dynamics on two sets of random walk bridges, $X^{n, i}_{j}$ for $j = 1, \ldots, k$ and $i \in \{ 1, 2\}$ as follows: each $X^{n, i}$ has endpoints from $x^{n}$ to $y^n$ and when restricted to $n^{-2} \Z \cap [a, b]$ has jumps of size $\pm 2 n^{-1}$. We can extend $X^{n,i}$ to $(a,b)$ by linearly interpolating between these values. We start both sets of random walks at their lowest possible configuration so that it is clear that $(X^{n,1}_{j})_0(u) \leq (X^{n, 2}_{j})_0(u)$ for all $j \in \{1, \ldots, k\}$ and $u \in [a,b]$.
		
		We now specify the continuous time Markov dynamics as follows. For each $s \in n^{-2} \Z \cap [a, b]$, $j \in \{ 1, \ldots, k\}$ and $m \in \{ -1, +1\}$ we have an exponential clock of rate one. When a clock rings, we attempt to alter both of the walks at time $s$ from $X^{n, i}_j(s)$ to $X^{n, i}_{j}(s) + 2 m n^{-1}$ for $i = 1, 2$. We call this proposed change $\tilde{X}^{n, i}$. The change is accepted if it is possible (i.e. the step sizes do not become too large) and according to the Metropolis rule. For $i \in \{ 1, 2\}$ define:
		$$
		R^{i} = \frac{W_{\mathbb{H}} (\tilde{X}^{n,i})}{W_{\mathbb{H}}  (X^{n,i})}
		$$
		Now sample a $U \sim \textrm{Unif}([0, 1])$. The change is accepted if $R^{i} \geq U$, where importantly the same uniform distribution is used for both Markov processes. Since these are both finite-state, irreducible Markov processes, these must converge to their stationary distribution in infinite time. In addition, as $n \to \infty$, it is clear that the corresponding stationary distributions converge to $\prob_{\Ham}$ as $n \to \infty$. Therefore, it remains to check that under this coupling, $X^{n, 1}_{j} \leq X^{n,2}_j$ for all $n, j$. 
		
		Suppose that the clock for $(s, j, m)$ rings and $X_{j}^{n,1}(s) = X_j^{n,2}(s)$. Note that there are only two cases we must consider which may affect the ordering: when both paths at $(s,j)$ looks like $\veeSymbol$ and $+$ rings, or both paths look like $\wedgeSymbol$ and $-$ rings. Suppose that we are in the former case (the second case is argued similarly) and we are trying to flip $\veeSymbol$ to $\wedgeSymbol$. We only need to check that $R^{1} \leq R^{2}$. Note that for each $i$
		\begin{align*}
			R^{i} &= \exp \Big(   \int_{s - n^{-2}}^{s + n^{-2}} \big( \mathcal{H}_{j - 1} (X^{n, i}_j (u) - X^{n, i}_{j - 1}(u) ) - \mathcal{H}_{j - 1} (\tilde{X}^{n, i}_{j} (u) - X^{n, i}_j(u)) \big) \, du   \Big) \\
			& \quad \cdots \times \exp \Big(   \int_{s - n^{-2}}^{s + n^{-2}} \big( \mathcal{H}_{j} (X^{n, i}_{j+1} (u) - X^{n, i}_{j}(u) ) - \mathcal{H}_{j} (\tilde{X}^{n, i}_{j+1} (u) - X^{n, i}_{j}(u)) \big) \, du   \Big).
		\end{align*}
		where the two terms in the product represent the change of the interaction strength between the two adjacent lines. We simplify notation by defining the following functions: 
		\begin{align*}
			d^i_k(u) &= X^{n, i}_{k + 1}(u) - X^{n, i}_{k}(u) , \qquad 
			\delta(u) = \tilde{X}^{n, i}_{j}(u) - X^{n, i}_j(u)
		\end{align*}
		Then we have:
		\begin{align*}
			R^{i} &= \exp \Big(  \int_{s - n^{-2}}^{s + n^{-2}} \big(\mathcal{H}_{j - 1} (d^i_{j - 1}(u))  - \mathcal{H}_{j - 1} (d^i_{j - 1} + \delta(u)) \big) \, du   \Big) \\
			& \cdots \times \exp \Big( \int_{s - n^{-2}}^{s + n^{-2}} \big( \mathcal{H}_{j} (d^i_j(u)) - \mathcal{H}_j(d^i_{j} (u) - \delta(u))   \big) \, du  \Big).
		\end{align*}
		Now as $X^{n, 1}_j \leq X^{n, 2}_j$, we have that $d^1_{j - 1}(u) \geq d^2_{j - 1}(u)$ and $d^{1}_{j}(u) \leq d^{2}_{j}(u)$. In addition, since each $\mathcal{H}_k$ is convex, we have:
		\begin{align*}
			\mathcal{H}_{j - 1} (d^1_{j - 1}(u))  - \mathcal{H}_{j - 1} (d^1_{j - 1} + \delta(u))  & \leq \mathcal{H}_{j - 1} (d^2_{j - 1}(u))  - \mathcal{H}_{j - 1} (d^2_{j - 1} + \delta(u))  \\
			\mathcal{H}_{j} (d^1_j(u)) - \mathcal{H}_j(d^1_{j} (u) - \delta(u) ) & \leq \mathcal{H}_{j} (d^2_j(u)) - \mathcal{H}_j(d^2_{j} (u) - \delta(u) )  
		\end{align*}
		This immediately implies that $R^{1} \leq R^2$, and therefore whenever $X^{n, 1}$ flips up, $X^{n, 2}$ will also flip up. A symmetric argument can be applied to the other case, concluding the proof. 
	\end{proof}

	\bibliographystyle{alpha}
	\bibliography{KPZ_Blk_Height_refs}

\end{document}